\documentclass[11pt]{amsart}
\usepackage[linkcolor=blue, citecolor = blue, linktocpage = true]{hyperref}
\hypersetup{colorlinks=true}

\usepackage{amsthm,amsfonts,amsmath,amssymb,amscd} 

\usepackage{verbatim}
\usepackage{float}
\usepackage{wrapfig}
\usepackage{mathtools}
\usepackage[color=orange!20, linecolor=orange]{todonotes}

\usepackage{enumitem}

\usepackage{ytableau}

\usepackage{tikz}
\usetikzlibrary{decorations.markings}
\tikzstyle{vertex}=[circle, draw, inner sep=0pt, minimum size=4pt]

\usetikzlibrary{arrows,automata}
\usepackage{tikz, tikz-3dplot, pgfplots}
\usetikzlibrary[positioning,patterns]
\usepackage{caption}
\usepackage{subcaption}
\usepackage{lscape}
\usepackage{rotating}

\newtheorem{theorem}{Theorem}[section]
\newtheorem{proposition}[theorem]{Proposition}
\newtheorem{lemma}[theorem]{Lemma}
\newtheorem{corollary}[theorem]{Corollary}

\theoremstyle{definition}
\newtheorem{definition}[theorem]{Definition}
\newtheorem{example}[theorem]{Example}

\theoremstyle{remark}
\newtheorem{remark}[theorem]{Remark}

\newcommand{\sgn}{\mathrm{sgn}}

\newcommand*{\Darrow}{\rotatebox[origin=c]{-90}{\(\Longrightarrow\)}}

\DeclarePairedDelimiter\floor{\lfloor}{\rfloor}

\title[Fundamental invariants of tensors, Latin hypercubes, \& Kronecker coefficients]
{Fundamental invariants of tensors, Latin hypercubes, and rectangular Kronecker coefficients} 

\author[Alimzhan Amanov \and Damir Yeliussizov]{Alimzhan Amanov \and Damir Yeliussizov}

\address{KBTU, Almaty, Kazakhstan}
\email{\href{mailto:alimzhan.amanov@gmail.com}{alimzhan.amanov@gmail.com}, \href{mailto:yeldamir@gmail.com}{yeldamir@gmail.com}}

\begin{document}

\begin{abstract}
We study polynomial SL-invariants of tensors, mainly focusing on fundamental invariants which are of smallest degrees. 
In particular, we prove that certain 3-dimensional analogue of the Alon--Tarsi conjecture on Latin cubes considered previously by B\"urgisser and Ikenmeyer, implies positivity of (generalized) Kronecker coefficients at  rectangular partitions and as a result provides values for degree sequences of fundamental invariants. 
\end{abstract}

\maketitle

\tableofcontents

\section{Introduction}
Consider tensors in $V = (\mathbb{C}^n)^{\otimes d}$ with the natural tensor action of the group $G = \mathrm{SL}(n)^{\times d}$ and let $\mathbb{C}[V^{}]^G$
be the ring of $G$-invariant polynomials  on $V$, graded as 
$
\mathbb{C}[V^{}]^G = \bigoplus_{m \ge 0}  \mathbb{C}[V^{}]^G_m,
$
where  $\mathbb{C}[V^{}]^G_m$ denotes the space of $G$-invariant homogeneous polynomials of degree $m$.

\vspace{0.5em}

By a {\it fundamental invariant} 
we call a polynomial in $\mathbb{C}[V^{}]^G$ of the smallest positive degree.\footnote{We use terminology as in \cite{bi}; more broadly, fundamental invariants may also refer to ring generators, e.g. \cite{brh13}, but here we mean the smallest invariants.} 
When $d$ is even, 
 there is a unique (up to a scalar) fundamental invariant of degree $n$, originally due to 
 Cayley \cite{cay}; it is a straightforward generalization of 
 the determinant, see Example~\ref{cayley1h}. 
 But when $d$ is odd, the situation turns out to be much more complicated; in particular, the degrees and descriptions of fundamental invariants 
 are not known in general. 

\vspace{0.5em}

This theme  
is closely related to the following problems: (i) positivity of generalized {Kronecker coefficients} at equal rectangular partitions, 
which determine dimensions of $\mathbb{C}[V^{}]^G_m$; 
and 
(ii) high-dimensional analogues of the {Alon--Tarsi conjecture} \cite{at} on Latin squares, which interestingly arise as nonvanishing evaluations of special fundamental invariants at unit tensor; such a 3-dimensional analogue is due to B\"urgisser and Ikenmeyer \cite{bi}.
Generally speaking, we prove that (ii) imply (i) which can then reveal most information on degrees of fundamental invariants.

\vspace{0.5em}

Let us now discuss our results in more detail.

\subsection{Degree sequences} 
Let us denote
$$
{\delta}_{d}(n) := \min \left\{ m : \dim \mathbb{C}[V^{}]^G_m > 0 \right\} \in n \mathbb{N}
$$
the degree of a fundamental invariant which is known to be a multiple of $n$, cf. Proposition~\ref{prop:span}.
As discussed above, $\delta_{d}(n) = n$ if $d$ is even. For the odd $d$ case we prove the following results. 

\begin{theorem}\label{thm:one}
Let $d \ge 3$ be odd. We have: 

(i) The bounds
\begin{align*}
n \lceil n^{1/(d - 1)} \rceil \le \delta_{d}(n) \le n^2.
\end{align*}

(ii) The lower bound is achieved at least in the following cases
\begin{align*}
\delta_d(n) = n \lceil n^{1/(d - 1)} \rceil\, \text{ for } n \in \{1, \ldots, 2^{d-1}\} \cup \{3^{d-1}, \ldots, 4^{d-1} \} \cup \{k^{d-1} -1, k^{d-1}  : k \in \mathbb{N}_{\ge 2}\},
\end{align*}
$n \in [k^{d - 1} - \sqrt{k}/2 + 1, k^{d-1}]$ if $k$ is even, 
and also 
$
\delta_d(n) \in \{3n, 4n \} 
\, \text{ for } n \in \{2^{d -1} + 1, \ldots, 3^{d-1}\}.
$
\end{theorem}

The smallest example when the degree is larger than the lower bound in Theorem~\ref{thm:one}~(i) is $\delta_3(7) = 7 \cdot 4$. 
In fact, $\delta_3(n) \ge n \lceil \sqrt{n}\, \rceil + n$ if $n = k^2 - 2$ is odd. 
Theorem~\ref{thm:one} generalizes and improves analogous results for $d = 3$ in \cite{bi}; this case was also studied recently in  \cite{LZX21}.


\vspace{0.5em}

In particular, our result shows that the function $\delta_d(n)$ has superlinear 
growth if $d$ is fixed. When $n > 1$ is fixed we prove that the degree sequence $\{\delta_d(n) \}_{d = 3,5,7,\ldots}$
is weakly decreasing and stabilizes to $2n$ for $d \ge 1 + \log_2 n$, see Corollary~\ref{cor56}.


\vspace{0.5em}

Parts of proofs are given in \textsection\,\ref{sec:degbound}, where we also obtain more general degree bounds for invariants of non-cubical tensors, and in \textsection\,\ref{sec:dimseq}, where we study the dimension sequences discussed below.

\subsection{Explicit invariants, Latin hypercubes, and Kronecker coefficients}
To study invariants and degree bounds, we use a general description of $G$-invariant polynomials spanning the space $\mathbb{C}[V^{}]^G_m$ \cite{widg}, see \textsection\,\ref{sec:delta} where we also introduce some basic combinatorial operations on these polynomials. 

For $n = k^{d-1}$ and odd $d \ge 3$ there is an explicit unique (up to a scalar) fundamental invariant $F_{d,k} \in \mathbb{C}[V]^{G}$ of degree $\delta_d(n) = k^{d}$, which is a generalization of the fundamental invariant of B\"urgisser and Ikenmeyer for $d = 3$ \cite{bi}, see \textsection\,\ref{sec:fund}. Its evaluation $F_{d,k}(I_n)$ on the unit tensor $I_{n}$ can be presented as a signed sum over $d$-dimensional {\it Latin hypercubes}, see Corollary~\ref{cor:fLatin} and \textsection\,\ref{sec:latin}.

 
Let us denote the corresponding dimension 
$$
g_{d}(n,k) := \dim \mathbb{C}[V]^G_{nk} = g(\underbrace{n \times k, \ldots, n \times k}_{d \text{ times}}), 
$$
which is also the generalized {\it Kronecker coefficient} at rectangular partitions $n \times k := (k^n)$. 

\vspace{0.3em}

Consider the following four statements:

$A_{d,k}$ : $F_{d,k}(I_{k^{d-1}}) \ne 0$. 
\qquad\qquad\qquad\qquad\qquad\qquad\qquad\qquad($d$-dimensional Alon--Tarsi)

$B_{d,k}$ : $g_{d}(n,k) > 0$ for all $n \le k^{d-1}$. \qquad\qquad\qquad\qquad\quad(rectangular Kronecker positivity)

$C_{d,k}$ : $\delta_{d}(n) = nk$ for all $(k-1)^{d-1} < n \le k^{d-1}$. \qquad\quad\,\,(degrees of fundamental invariants)

$C'_{d,k}$ : $\delta_{d}(n) \in \{nk, n(k+1) \}$ for all $(k-1)^{d-1} < n \le k^{d-1}$.

\vspace{0.3em}

We prove that these statements are in fact related. Namely, 
for every odd $d \ge 3$ the following implications hold: 

\begin{center}
\begin{tabular}{c c l c l}
$A_{3,k}$ & $\Longrightarrow^{\text{Thm~\ref{thm81}}}$ & $B_{3,k}$ & & \\
& & $\Darrow~{\scriptsize\text{Thm~\ref{gthm}(iii)}}$ & & \\
$A_{d,k}$ & $\Longrightarrow^{\text{Thm~\ref{thm81}}}$ & $B_{d,k}$ & $\Longrightarrow^{\text{Lemma~\ref{lemmagd}}}$ & $C_{d,k}$ and $C'_{d,k-1}$ 
\end{tabular}
\end{center}

The statement $A_{d,k}$ (which is false for odd $k$, cf. Corollary~\ref{cor:zeroev}) is a $d$-dimensional analogue of the celebrated {\it Alon--Tarsi conjecture} on Latin squares \cite{at}. 
Hence, conditionally on $A_{3,k}$ for even $k$, we obtain the positivity of Kronecker coefficients $B_{d,k}$ and then degree values $C_{d,k}$ and $C'_{d,k-1}$ for all odd $d \ge 3$. 
We show that $A_{d,2}$ is true for all $d \ge 2$, see Proposition~\ref{at2nz}; in \cite{bi} it is noted that $A_{3,k}$ is true for $k = 2,4$. We also provide computations that $B_{3,k}$ is true for $k = 2,4$, see Appendix~\ref{appa}. As a result we get the true statements $C_{d,2}$, $C_{d,4}$, $C'_{d,3}$ for all odd $d \ge 3$ which is reflected in Theorem~\ref{thm:one}~(ii). 
It also seems that $B_{d,k}$ holds for all odd $d \ge 5$ and all $k$. Regarding the degree sequence $\delta_d(n)$ we conjecture more precisely 
that for odd $d \ge 3$ it matches the lower bound in Theorem~\ref{thm:one}~(i) with 
the only exceptional cases $d = 3$ and odd $n = k^2 - 2$ for which it is larger by $n$.

We also show that the usual Alon--Tarsi conjecture on Latin squares for $k$ implies a weaker statement: that $g_d(n, k) > 0$ for $n \le k$ (see Corollary~\ref{cor:atk}), and hence $\delta_{d}(n) = nk$ for $k^{d - 1} - k \le n \le k^{d-1}$ (cf. Remark~\ref{rem:atk}). 

Furthermore, unconditionally we prove that  
$g_d(n, k) > 0$ for all $n \le \sqrt{k}/2 - 1$ if $k$ is even (Theorem~\ref{uncon}), and hence $\delta_d(n) = nk$ for $k^{d -  1} - \sqrt{k}/2  + 1 \le n \le k^{d-1}$.


\vspace{0.5em}

In \textsection\,\ref{sec:dimseq} 
we establish some properties of the dimension sequences $g_{d}(n,k)$ including the symmetry $g_{d}(n,k) = g_{d}(k^{d - 1} - n, k)$ and the implication $B_{3,k} \implies B_{d,k}$  (cf. Theorem~\ref{gthm}) on positivity 
using various properties of Kronecker coefficients. In  \textsection\,\ref{sec:hwv0} we prove the implication $A_{d,k} \implies B_{d,k}$ (cf. Theorem~\ref{thm81}) by describing some explicit constructions of highest weight vectors combined with properties of Latin hypercubes  obtained in \textsection\,\ref{sec:latin}. 
In particular, a core ingredient in the proof is presented by an important $k$-form $\omega$ in the anti-symmetric space $\bigwedge^k (\mathbb{C}^k)^{\otimes d}$, which is dual to Cayley's first hyperdeterminant (cf. Remark~\ref{remod}) and whose power expansions surprisingly record all evaluations of $G$-invariant polynomials at unit tensors (cf.~Theorem~\ref{thm:omega}). 

\begin{remark}
An important known $G$-invariant is the {\it geometric hyperdeterminant} \cite{gkz} whose degree sequence has an exponential growth even for $d = 3$ (cf. \cite[A176097]{oeis}).  
\end{remark}

\begin{remark}
Hilbert  \cite{hilb1, hilb2} proved that the ring of invariants $\mathbb{C}[V^{}]^G$ 
is finitely generated for linearly reductive groups $G$. Note that degrees of generators  can be way larger than  of smallest invariants, see \cite{derk} for general upper bounds and \cite{dm} for exponential lower bounds on related tensor spaces. 
\end{remark}

\section{Preliminaries}
\subsection{Tensors} The group $G = \mathrm{SL}(n_1) \times \cdots \times \mathrm{SL}(n_d)$ naturally acts on the space of tensors $V = \mathbb{C}^{n_1} \otimes \cdots \otimes \mathbb{C}^{n_d}$ by $(g_1,\ldots, g_d) \cdot v_1 \otimes \cdots \otimes v_d = g_1 v_1 \otimes \ldots \otimes g_d v_d$ for $g_i \in \mathrm{SL}(n_i)$ and extended multilinearly. We denote 
$$
\mathrm{Inv}(n_1,\ldots, n_d) := \mathbb{C}[V]^G = \left\{P \in \mathbb{C}[V^{}] :  P(g \cdot v) = P(v) ~ \forall g \in G , v \in V^{} \right\}
$$
the ring of $G$-invariant polynomials on $V$ which is graded
$$
\mathbb{C}[V]^G = \bigoplus_{m \ge 0} \mathbb{C}[V]^G_m,
$$
where 
$\mathbb{C}[V^{}]^G_m$ is the space of $G$-invariant homogeneous polynomials of degree $m$ which we denote by 
$$
\mathrm{Inv}(n_1,\ldots, n_d)_{m} := \mathbb{C}[V]^G_m.
$$ 
In the case $n_1 = \ldots = n_d = n$ we use the notation 
$$
\mathrm{Inv}_d(n) := \mathrm{Inv}(n,\ldots, n) 
\text{ and } 
\mathrm{Inv}_d(n)_m := \mathrm{Inv}(n,\ldots, n)_m.
$$
It is known that if the space $\mathrm{Inv}(n_1,\ldots, n_d)_{m}$ is nonzero then $n_i$ divides $m$. 

The elements of $V$ written in a fixed basis correspond to hypermatrices $(X_{i_1,\ldots, i_d})$ indexed by $(i_1,\ldots, i_d) \in [n_1] \times \cdots \times [n_d]$ and we shall usually identify tensors in $V$ with corresponding hypermatrices. 

\subsection{Partitions} 
	A {\it partition} is a sequence $\lambda = (\lambda_{1}, \ldots, \lambda_{\ell})$ of positive integers $\lambda_1 \ge \cdots \ge \lambda_{\ell}$, where $\ell(\lambda) = \ell$ is its {\it length}. The {\it size} of $\lambda$ is $\lambda_1 + \ldots + \lambda_{\ell}$. Every partition $\lambda$ can be represented as the {\it Young diagram} $\{(i,j)  : i \in [1, \ell], j\in [1,\lambda_{i}]\}$. We denote by $\lambda'$ the {\it conjugate} partition of $\lambda$ whose diagram is transposed. We denote rectangular partitions as $a \times b := (\underbrace{b, \ldots, b}_{a~\text{times}})$. For $\lambda \subseteq a \times b$ we also denote by $a \times b -_a \lambda := (b - \lambda_a, \ldots, b - \lambda_1)$ the {\it complementary} partition of $\lambda$ inside $a \times b$. 
	
\subsection{Kronecker coefficients}
	For partitions $\lambda, \mu, \nu$ of the same size $n$, the {\it Kronecker coefficient} $g(\lambda, \mu, \nu)$ is defined as the multiplicity of $[{\nu}]$ in the tensor product decomposition $[{\lambda}] \otimes [{\mu}]$, where $[{\alpha}]$ denotes the irreducible representation of $S_{n}$ indexed by partition $\alpha$. 	
	Here is the list of some known properties of Kronecker coefficients 
	which will be useful for us.
	\begin{lemma}\label{kron}
		The following properties hold:
		\begin{enumerate}[label=(\alph*)]
			\item ($S_3$ symmetry) $g(\lambda,\mu,\nu) = g(\lambda, \nu, \mu) = \cdots$ is invariant under permutations of $\lambda,\mu,\nu$.
			\item (Conjugation) $g(\lambda,\mu,\nu) = g(\lambda,\mu',\nu')$. 
			\item (Trivial charachter) $g(1 \times k, \lambda, \mu) = \delta_{\lambda,\mu}$, $g(k \times 1, \lambda, \mu) = \delta_{\lambda,\mu'}$.
			\item (Square positivity  \cite{bb04}) $g(k \times k, k \times k, k \times k) > 0$ for every $k \in \mathbb{N}$.
			\item (Rectangular identities \cite{val09}, \cite[Cor.~4.4.14-15]{iken13}) Let $a,b,c \in \mathbb{N}$.\\ 
			If $\lambda,\mu,\nu$ satisfy $\ell(\lambda) \le a, \ell(\mu) \le b, \ell(\nu) \le ab$ then
			$$g(\lambda,\mu,\nu) = g(\lambda + a\times bc, \mu + b\times ac, \nu + ab\times c).$$ 
			 If $\lambda,\mu,\nu$ satisfy $\lambda \subseteq bc\times a, \mu \subseteq ac\times b, \nu \subseteq ab\times c$ then
			$$g(\lambda,\mu,\nu) = g(bc \times a -_{bc} \lambda, ac \times b -_{ac} \mu, ab \times c -_{ab} \nu).$$
			\item (Size bounds \cite{dvir93}) $\max \{ \lambda_{1} : g(\lambda,\mu,\nu) > 0 \} = |\mu \cap \nu|$. In particular, if $g(\lambda,\mu,\nu) > 0$ then $\ell(\lambda) \le \ell(\mu)\ell(\nu)$.
			\item (Semigroup \cite{chm}) If $g(\lambda,\mu,\nu) > 0$ and $g(\alpha,\beta,\gamma) > 0$ then $g(\lambda+\alpha,\mu+\beta,\nu+\gamma) > 0$.
		\end{enumerate}
	\end{lemma}	 

	For partitions $\lambda^{(1)}, \ldots, \lambda^{(d)}$ of size $n$, the generalized {Kronecker coefficient} $g(\lambda^{(1)}, \ldots, \lambda^{(d)})$ is defined as the multiplicity of $[{\lambda^{(d)}}]$ in the tensor product decomposition  
$
[{\lambda^{(1)}}] \otimes \cdots \otimes [{\lambda^{(d-1)}}]. 
$

The following fact is well known, e.g. \cite{luq}.
\begin{proposition}
Let $k_i = m/n_i$. We have 
$$
\dim \mathrm{Inv}(n_1,\ldots, n_d)_m = g(n_1 \times k_d, \ldots, n_d \times k_d).
$$
\end{proposition}

\section{Invariant polynomials}\label{sec:delta}
We denote $[n] := \{1,\ldots, n\}$.
\begin{definition}
  Let 
  $M = nk$ and $P(n, M)$ be the set of  set partitions of $[M]$ into $k$ blocks each of size $n$. 
  For a set partition $A \in P(n,M)$ 
  given by $A_1 \cup \cdots \cup A_k = [M]$ 
 with $A_i = \{a_{i,1} < \cdots < a_{i,n}\}$ for all $i \in [k]$ and a map $\sigma : [M] \to [n]$ define the {\it sign} as follows:
 $$
  	\sgn_{A}(\sigma) = \sgn(\sigma(a_{1,1}),\ldots,\sigma(a_{1,n})) \cdots \sgn(\sigma(a_{k,1}),\ldots,(\sigma(a_{k,n})), 
 $$
   where $\sgn(a_{1},\ldots,a_{n}) \in \{0, \pm 1\}$ is the usual sign 
   if $(a_1,\ldots, a_n)\in S_n$ is a permutation, and $0$ otherwise. 
\end{definition}

\begin{definition}
  Let $X \in V$ given as a hypermatrix. Define the polynomials $\{ \Delta_S\}$ 
  indexed by $d$-tuples of set partitions $S  = (S^{1}, \ldots, S^{d}) \in P(n_1, M) \times \cdots \times P(n_d, M)$ as
  \begin{align}\label{hdetdef}
		\Delta_S(X) := \sum_{\sigma_1 : [M] \to [n_1]} \cdots \sum_{\sigma_d : [M] \to [n_d]} \sgn_{S^{1}}(\sigma_1)\cdots\sgn_{S^{d}}(\sigma_d)
      \prod_{i=1}^M X_{\sigma_1(i),\ldots,\sigma_d(i)}.
  \end{align}
\end{definition}

These polynomials are $G$-invariant and describe the space of invariants as follows.
 
\begin{proposition}[{\cite[Prop.~3.10]{widg}}]\label{prop:span}
	The polynomials $\{ \Delta_{S} \}$ indexed by $d$-tuples of set partitions ${S  \in P(n_1, M) \times \cdots \times P(n_d, M)}$ span the space $\mathrm{Inv}(n_{1},\ldots,n_{d})_{M}$. 
	If $n_{i}$ does not divide $M$ for some $i$, this space is empty.
\end{proposition}


\begin{remark}
    In \cite{widg} and \cite[Ex.~7.18]{widg2} these $\Delta$ polynomials are presented in a slightly different form; they are indexed by $d$-tuples of permutations of $[M]$ (i.e. if set partitions $S$ are not ordered in increasing order initially). The versions differ up to a sign. 
\end{remark}

\begin{remark}
While the form of $\Delta$ polynomials is explicit, 
it can be difficult to understand some basic questions about them, such as when $\Delta_S$ is a nonzero polynomial or if its evaluation on unit tensor is nonzero.
\end{remark}

\subsection{Balanced tables}
It is convenient to represent the indices of $\Delta$ polynomials via what we call {\it balanced} tables. 
We say that a $d \times M$ table is {\it balanced} if its row $i$ contains every element of $[M/n_i]$ exactly $n_i$ times for all $i \in [d]$. Given a balanced table $T = (T_{i j})$ we construct $S = (S^1, \ldots, S^d) \in P(n_1, M) \times \cdots \times P(n_d, M)$ as follows: if $T_{i j} = \ell$ then $j$ belongs to the $\ell$-th block of the set partition $S^{i}$. We then write $\Delta_T$ for $\Delta_S$. When using the expansion \eqref{hdetdef} we shall also denote 
$$
\sgn_{T}(\sigma_1, \ldots, \sigma_d) := \sgn_{S^{1}}(\sigma_1)\ldots\sgn_{S^{d}}(\sigma_d), \qquad \sgn_{T_i} := \sgn_{S^i}.
$$

For example, the balanced table 
		\begin{equation*}
			T_{} = \begin{pmatrix}
			1 1 2 2 3 3 \\
			1 1 1 2 2 2\\
			1 1 2 1 2 2
		\end{pmatrix}
	\end{equation*}
corresponds to the triple $S = (S^{1}, S^{2}, S^{3}) \in P(2, 6) \times P(3, 6) \times P(3,6)$ of the set partitions
\begin{align*}
S^1 &= \{1,2 \} \cup \{3,4\} \cup \{5,6\},\\
S^2 &= \{1,2,3\} \cup \{4,5,6\}, \\
S^3 &= \{1,2,4\}  \cup \{3,5,6\}.
\end{align*}

Note that some balanced tables may produce the same $d$-tuple $S$ and hence they may be regarded as {\it equivalent}. In particular, if in any row of $T$ we switch all the values $a$ and $b$ from $[M/n_i]$, the resulting $d$-tuple $S$ does not change (since blocks in set partitions are unordered). 

\subsection{Some examples} Let us show some examples of the invariant polynomials $\Delta$.

\begin{example}[Cayley's first hyperdeterminant]\label{cayley1h}
 A simple-looking fundamental invariant arises as follows. Let 
 $T$ be the $d \times n$ balanced table of all ones
\begin{align*}
T_{} = \begin{pmatrix}
	1 1 \ldots 1 \\
	\ldots\\
	1 1 \ldots  1
\end{pmatrix}.
\end{align*}
Then we have 
$$
\Delta_T(X) = \sum_{\sigma_1,\ldots,\sigma_d \in S_n} \sgn(\sigma_{1}\cdots\sigma_d) \prod_{i=1}^n X_{\sigma_1(i), \sigma_2(i),\ldots,\sigma_d(i)}.
$$ 
This function is nontrivial when $d$ is even, otherwise it is identically $0$. This invariant was introduced by Cayley \cite{cay}.
\end{example}

\begin{example}[Cayley's second hyperdeterminant]
Let $d = 3$, $n = 2$ and consider the following balanced table:
\begin{align*}
T_{} = \begin{pmatrix}
	1 1 2 2\\
	1 1 2 2\\
	1 2 1 2
	\end{pmatrix}.
\end{align*}
Then we have 
\begin{align*}
      -\frac{1}{2} \Delta_T(X) 
      =  & \, X_{111}^2 X_{222}^2 + X_{112}^2 X_{221}^2 +
        X_{121}^2 X_{212}^2 + X_{211}^2 X_{122}^2  \\
        &-2(X_{111}X_{112}X_{221}X_{222}+X_{111}X_{121}X_{212}X_{222}
        + X_{111}X_{211}X_{122}X_{222}\\
        &+X_{112}X_{121}X_{212}X_{221}
        + X_{112}X_{211}X_{122}X_{221}+X_{121}X_{211}X_{122}X_{212}) \\
        &+ 4 (X_{111}X_{122}X_{212}X_{221} + X_{112}X_{121}X_{211}X_{222}).
\end{align*}
This function is the unique (up to scalar) fundamental invariant of $\mathrm{Inv}(2,2,2)$ and in fact generates this ring 
\cite[Theorem~13]{bbs12}. Originally it was also computed by Cayley \cite{cay2} and it is the simplest example of the {\it geometric hyperdeterminant} studied in \cite{gkz}.
\end{example}

\begin{example}[$2\times2\times3$ geometric hyperdeterminant] 
Let $d = 3$, $(n_{1},n_{2},n_{3}) = (2, 2, 3)$ and consider the following balanced table: 
\begin{align*}
T_{} = \begin{pmatrix}
	1 1 2 2 3 3\\
	1 1 2 2 3 3\\
	1 2 1 2 1 2
	\end{pmatrix}.
\end{align*}
Then $\Delta_T$ is the unique (up to scalar) fundamental invariant of the ring $\mathrm{Inv}(2, 2, 3)$, and $-\frac{1}{12} \Delta_T$ gives the geometric hyperdeterminant. See \cite{br12} for related computations. 
\end{example}

\begin{example}[$3\times 3 \times 3$ ring of invariants]
The ring $\mathrm{Inv}(3,3,3)$ is generated by three homogeneous polynomials 
of degrees $6,9$ and $12$, see \cite{brh13}. 
As generators, 
one can use the invariant polynomials indexed by the following balanced tables:
\begin{align*}
 T_{1} = \begin{pmatrix}
	1 1 1 2 2 2\\
	1 1 2 2 2 1\\
	1 2 2 2 1 1
	\end{pmatrix},\quad
 T_{2} = \begin{pmatrix}
	1 1 1 2 2 2 3 3 3\\
	1 1 2 2 2 3 3 3 1\\
	1 2 2 2 3 3 3 1 1
	\end{pmatrix},\quad
 T_{3} = \begin{pmatrix} 
	1 1 1 2 2 2 3 3 3 4 4 4\\
	1 1 1 2 2 2 3 3 3 4 4 4\\
	1 2 3 1 2 4 1 3 4 2 3 4
	\end{pmatrix}. %
\end{align*}
(To compare with generators in \cite{brh13}, $\Delta_{T_1}, \Delta_{T_2}$ differ up to scalar, and 
up to $\Delta_{T_3} + c \Delta_{T_1}^2$.)
Note that the geometric hyperdeterminant of format $3\times3\times 3$ has degree $36$ and can be written as $P(\Delta_{T_1}, \Delta_{T_2}, \Delta_{T_3})$ for some polynomial $P$, 
see \cite{bro14} for an explicit presentation. In expansion of $P$, each monomial of type $\Delta_{T_1}^{\alpha} \Delta_{T_2}^{\beta} \Delta_{T_3}^{\gamma}$ can be written as a $3\times36$ balanced table: horizontally concatenate $T_{1}$ $\alpha$ times, $T_{2}$ $\beta$ times and $T_{3}$ $\gamma$ times, so that each concatenated table is shifted in numbers (see \textsection\,\ref{horcon} for this operation). This way the geometric hyperdeterminant can be written as a linear combination of the $\Delta$ polynomials. 
\end{example}

\subsection{Basic properties} We now show some basic properties of the $\Delta$ polynomials.
\begin{proposition}[Relative GL-invariance]
    Let $X \in V$, $(A_{1},\ldots,A_{d}) \in \mathrm{GL}(n_{1}) \times \cdots \times \mathrm{GL}(n_{d})$. For $S \in P(n_1, M) \times \cdots \times P(n_d, M)$ we have
    \begin{align*}
        \Delta_S((A_{1},\ldots,A_{d}) \cdot X) = \det(A_{1})^{M/n_{1}}\cdots\det(A_{d})^{M/n_{d}} \,  \Delta_S(X).
    \end{align*}
\end{proposition}
\begin{proof}
  Since $(A_{1},\ldots,A_{d}) \cdot X = ((A_{1},\ldots,I) \cdot \ldots \cdot (I,\ldots,A_{d}))\cdot X$ it is enough to check the identity for an action of one component, say for $(A_{1},I,\ldots,I)$. Denote $A = A_{1}$ and $Y = (A,I,\ldots, I) \cdot X$. Let $S = (S^1, \ldots, S^d)$. Then
  \begin{align*}
      \Delta_S(Y)  
      &= \sum_{\sigma_{1},\ldots,\sigma_{d}}
      	\sgn_{S^1}(\sigma_1) \cdots \sgn_{S^d}(\sigma_d)
		\prod_{i=1}^M \left( 
		\sum_{\ell_{i}=1}^{n_{1}} 
			A_{\sigma_1(i),\ell_{i}}
			X_{\ell_{i},\ldots,\sigma_{d}(i)} \right) 
	\\ &= \sum_{\ell:[M] \to [n_{1}]} 
		\sum_{\sigma_{1},\ldots,\sigma_{d}}
		\sgn_{S^1}(\sigma_1) \cdots \sgn_{S^d}(\sigma_d) 
		\prod_{i=1}^M 
			A_{\sigma_1(i),\ell({i})} X_{\ell({i}),\ldots,\sigma_{d}(i)} 
	\\ &= \sum_{\ell:[M] \to [n_{1}]} 
		\sum_{\sigma_{2},\ldots,\sigma_{d}}
		\sgn_{S^{2}}(\sigma_{2}) \cdots \sgn_{S^{d}}(\sigma_{d})
		\prod_{i=1}^M  X_{\ell(i),\ldots,\sigma_{d}(i)} 
		 \times \left(
			\sum_{\sigma_{1}} 
				\sgn_{S^{1}}(\sigma_{1}) A_{\sigma_1(i),\ell(i)}
		\right)
  \end{align*}
  where $\sigma_{i}: [M] \to [n_{i}]$ for $i \in [d]$. Note that $\sigma_{1}$ splits into $k_{1}=M/n_{1}$ subpermutations according to $S^{1}$ when $\sgn_{S^{1}}(\sigma_{1})\neq 0$; hence the last sum factors into $k_{1}$ determinants of matrices formed by the columns $\ell(1),\ldots,\ell(M)$ according to $S^{1}$. Since permutation of columns alters the sign of the determinant by 
  its sign and vanishes when two columns are equal, we have
  \begin{align*}
      \Delta_S(Y) = \det(A)^{k_{1}} \sum_{\ell, \sigma_{2},\ldots,\sigma_{d}} \sgn_{S^{1}}(\ell) \ldots \sgn_{S^{d}}(\sigma_{d}) \prod_{i=1}^{M} X_{\ell(i),\ldots,\sigma_{d}(i)} = \det(A)^{k_{1}} \Delta_{S}(X)
  \end{align*}
   and the proof follows. 
\end{proof}

A hypermatrix $X$ corresponding to a tensor in $V$ has $n_{\ell}$ {\it parallel slices} in the direction $\ell \in [d]$ given by fixing the $\ell$-th coordinate $i_{\ell} \in [n_{\ell}]$. 

\begin{corollary}
The polynomials $\Delta_{S}$ satisfy: 
\begin{enumerate}
  \item Exchanging any two parallel slices in direction $i$ changes $\Delta_{S}$ by $(-1)^{M/n_{i}}$.
  \item  $\Delta_{S}$ is a homogeneous polynomial in the entries of each slice. The degree of homogeneity is the same for parallel slices in direction $i$ and is equal to $M/n_{i}$. 
  \item  $\Delta_{S}$ does not change if we add to any slice a scalar multiple of a parallel slice.
  \item $\Delta_{S} = 0$ if there are two parallel slices proportional to each other. 
\end{enumerate}
\end{corollary}

\begin{remark}
For even $d$ and minimal degree $M = n_1 = \cdots = n_d$, such properties in fact characterize Cayley's first hyperdeterminant, see \cite[Prop.~3.2]{ay}.
\end{remark}

\subsection{Column swap}
Let $T$ be a balanced $d \times M$ table. For $i,j\in[M]$ denote by $(i,j) T$ the table resulting in the swap of the $i$-th and $j$-th columns of $T$.

\begin{lemma}\label{colswap}
	Let $i < j \in [M]$. Then 
	$$
		\Delta_{T} = (-1)^{\ell} \Delta_{(i,j)  T}
	$$
	where
	\begin{align}\label{colsign}
	\ell = \sum_{m=1}^{d}
	\#\{c\in[i,j]: T_{m, i} = T_{m, c} \} + \#\{ c\in[i,j]: T_{m, j} = T_{m, c} \} - \delta_{T_{m, i} T_{m,j}}
	\end{align}
\end{lemma}
\begin{proof}
	Each monomial in the expansion \eqref{hdetdef} of $\Delta_{T}$ can be bijectively mapped to the same monomial of $\Delta_{(i,j)T}$. 
	For a tuple of maps $\sigma = (\sigma_1,\ldots, \sigma_d)$ as in the expansion \eqref{hdetdef} we associate another tuple of maps $\tau = (\tau_1,\ldots, \tau_d)$ given by $\tau_{m}  = \sigma_{m} \circ (i, j)$ (here $\circ$ is the composition) for each $m \in[d]$. 
	Clearly, $\{\sigma(1),\ldots,\sigma(M)\} = \{\tau(1),\ldots,\tau(M)\}$ as sets of vectors. Denote 
	\begin{align} \label{lm}
	\ell_{m} = \#\{c \in(i,j): T_{m, i} = T_{m, c} \} + \#\{c \in(i,j): T_{m, j} = T_{m, c}\} - \delta_{T_{m, i} T_{m,j}}
	\end{align}
	and let us show that $\sgn_{T_{m}}(\sigma_{m})= (-1)^{\ell_{m}}\, \sgn_{(i,j)\circ T_{m}}(\tau_{m})$ where $T_{m}$ is the $m$-th row of $T$. Let $x = T_{m,i}$ and $y = T_{m,j}$. Except for $x$-th and $y$-th sub-permutations, the relative order of other sub-permutations does not change, hence their contributions remains the same. If $x = y$ then the number of inversions of $x$-th sub-permutations in $\sigma_{m}$ and $\tau_{m}$ differ by one, and $\ell_{m}$ is odd, as needed. If $x\neq y$ then the number of inversions in $x$-th (or $y$-th) sub-permutation changes by the number of elements of this permutation that lie in the range $[i,j]$ of $T_{m}$,  which is exactly the terms of \eqref{lm}, as needed. Finally, we have 
	$$\sgn_{T}(\sigma) = \prod_{m=1}^{d}\sgn_{T_{m}}(\sigma_{m}) = \prod_{m=1}^{d}(-1)^{\ell_{m}}\sgn_{(i,j)\circ T_{m}}(\tau_{m})= (-1)^{\ell} \sgn_{(i,j) T}(\tau)$$
	and the proof follows.	
\end{proof}

\begin{corollary}\label{identicalcol}
	Let $d$ be {odd}. If $T$ has two equal columns then $\Delta_{T} = 0$.
\end{corollary}
\begin{proof}
 	Suppose $i \neq j$ are indices of two equal columns in $T$. Then for each row $m\in[d]$ we have $T_{m, i} = T_{m, j}$. 
	Then we have $\Delta_{T} = (-1)^{\ell} \Delta_{(i,j)_{}T} = (-1)^{\ell} \Delta_{T}$ and
	$$
		\ell 
		= \sum_{m = 1}^d (2 \#\{c\in[i,j]: T_{m, i} = T_{m, c} \}  - 1)
	$$
	which is an odd number. 
	Hence $\Delta_{T}=-\Delta_{T}$ and the claim follows.
\end{proof}

\begin{remark}
These properties show that we may disregard the order of columns of $T$ and view it as a set of its columns (as vectors) by sorting them in lexicographical order; this point of view will be especially useful in \textsection\,\ref{sec:latin}.
\end{remark}

\begin{remark}
Let us note that having distinct columns of $T$ is a necessary condition for $\Delta_T$ to be nonzero but not sufficient. For example, we have 
$$
T = \begin{pmatrix}
1 1 2 2 3 3\\ 
1 1 2 2 3 3\\ 
1 2 2 3 1 3
\end{pmatrix} \qquad \Delta_T = 0.
$$
Note that in this case $d = 3, n = 2$ and there are no invariants of degree $6$. 
\end{remark}

\begin{remark}
	Swapping two rows $(i,j)_r$ of $T$ affects the polynomial $\Delta_{T}$ in the following way:
	$$
		\Delta_{(i,j)_{r}T}(X) = \Delta_{T}(X^{t(i,j)})
	$$
	where $X^{t(i,j)}$ is the result of transposing $X$ along directions $i$ and $j$. In the case of Cayley's first hyperdeterminant, we have $\Delta(X) = \Delta(X^{t(i,j)})$ (same as for ordinary determinant), but for general functions $\Delta$ this property does not hold. For example, sometimes a swap of two rows in a balanced table can be replaced with column swaps for odd $\ell$ and the result will differ by sign.
\end{remark}

\begin{remark}
Let us also note that the invariant polynomials $\Delta$ from the same space can be equivalent up to a scale $\ne \pm 1$, which is different to previous instances of equivalences. For example, let $d =  n = k = 3$ and consider the following balanced tables: 
\begin{align*}
T_1 = 
\begin{pmatrix}
1 1 1 2 2 2 3 3 3  \\
1 1 2 1 2 3 2 3 3 \\
1 2 1 \textcolor{red}{2} 2 \textcolor{red}{3 3} 1 3 
\end{pmatrix}, 
\quad 
T_2 = 
\begin{pmatrix}
1 1 1 2 2 2 3 3 3  \\
1 1 2 1 2 3 2 3 3 \\
1 2 1 \textcolor{red}{3} 2 \textcolor{red}{2 3} 1 3 
\end{pmatrix},
\quad
T_3 = 
\begin{pmatrix}
1 1 1 2 2 2 3 3 3  \\
1 1 2 1 2 3 2 3 3 \\
1 2 1 \textcolor{red}{3} 2 \textcolor{red}{3 2} 1 3 
\end{pmatrix},
\end{align*}
where the red entries indicate the difference between the tables.
Then we have 
$$
\Delta_{T_1} = -4 \Delta_{T_2} = 2 \Delta_{T_3} \ne 0.
$$
\end{remark}

\subsection{Horizontal concatenation}\label{horcon}
Let $T_{1}$ and $T_{2}$ be $d \times M_{1}$ and $d \times M_{2}$ balanced tables, respectively. Denote by $(T_{1}T_{2})$ the $d \times (M_1 + M_2)$ balanced obtained by the {\it horizontal concatenation} of $T_1$ and $T_2$ so that each entry of $T_{2}$ in row $i$ is increased by $M_{1}/n_{i}$ (to make the entries in $T_{1}$ and $T_{2}$ different). For example,
\begin{align*}
 T_{1} = \begin{pmatrix}
	1 1 1 2 2 2\\
	1 1 2 1 2 2\\
	1 2 2 1 1 2
	\end{pmatrix},\quad
 T_{2} = \begin{pmatrix}
	1 1 1 2 2 2 3 3 3\\
	1 1 2 1 2 3 2 3 3\\
	1 2 1 2 3 1 3 2 3
	\end{pmatrix},\quad
 (T_{1} T_{2}) = \begin{pmatrix}
	1 1 1 2 2 2 {\bf 3 3 3 4 4 4 5 5 5}\\
	1 1 2 1 2 2 {\bf 3 3 4 3 4 5 3 5 5} \\
	1 2 2 1 1 2 {\bf 3 4 3 4 5 3 5 4 5}
	\end{pmatrix}
\end{align*}
It is easy to see that the product of elements of the ring $\mathrm{Inv}(n_{1},\ldots,n_{d})$ can be written via horizontal concatenations.
\begin{proposition}\label{hconcat}
	Let $\Delta_{T_{1}}, \Delta_{T_{2}} \in \mathrm{Inv}(n_{1},\ldots,n_{d})$. Then
	$$
		\Delta_{T_{1}} \cdot \Delta_{T_{2}} = \Delta_{(T_{1}T_{2})} \in \mathrm{Inv}(n_{1},\ldots,n_{d}).
	$$
\end{proposition}

\subsection{Vertical concatenation}
Let $T_1$ and $T_2$ be $\ell \times M$ and $(d - \ell) \times M$ balanced tables respectively so that $\Delta_{T_1} \in \mathrm{Inv}(n_1,\ldots, n_{\ell})$ and $\Delta_{T_2} \in \mathrm{Inv}(n_{\ell+1}, \ldots, n_{d})$. 
Denote by $T = \left(\substack{T_{1}\\T_{2}}\right)$ the $d \times M$ balanced table obtained by the {\it vertical concatenation} of $T_1$ and $T_2$. Note that $\Delta_T \in \mathrm{Inv}(n_1,\ldots, n_d)$.

\begin{proposition}\label{productlemma}
Let $Y \in V_1 = \mathbb{C}^{n_1} \otimes \cdots \otimes \mathbb{C}^{n_{\ell}}$, $Z \in V_2 = \mathbb{C}^{n_{\ell + 1}} \otimes \cdots \otimes \mathbb{C}^{n_{d}}$. We have
$$
	\Delta_{T}(Y \otimes Z) = \Delta_{T_{1}}(Y) \cdot \Delta_{T_{2}}(Z),
$$
where 
$Y \otimes Z \in V$ is the outer tensor product. 
\end{proposition}
 \begin{proof} 
Let $X = Y \otimes Z$. We have
	\begin{align*}
		\Delta_{T}(X) &= \sum_{\sigma_{1},\ldots,\sigma_{d}} \sgn_{T}(\sigma_1, \ldots, \sigma_d) \prod_{i=1}^{M} X_{\sigma_{1}(i),\ldots,\sigma_{d}(i)}\\
		&= \sum_{\sigma_{1},\ldots,\sigma_{d}} \sgn_{T_{1}}(\sigma_{A})\, \sgn_{T_{2}}(\sigma_{B}) \prod_{i=1}^{M} Y_{\sigma_{A}(i)} Z_{\sigma_{B}(i)}
		= \Delta_{T_{1}}(Y) \cdot \Delta_{T_{2}}(Z),
	\end{align*}
	where for a subset $A=\{a_{1}<\ldots<a_{k}\}\subseteq[d]$ we denote $\sigma_{A} = (\sigma_{a_{1}},\ldots,\sigma_{a_{k}})$ and $\sigma_{A}(i) = (\sigma_{a_{1}}(i),\ldots,\sigma_{a_{k}}(i))$. 
\end{proof}

\begin{corollary}\label{nullity}
	Assume $\Delta_{T_1}$ is identically $0$. Then the restriction of $\Delta_{T}$ on $V_1 \otimes V_2$ is identically $0$. 
\end{corollary}
\begin{corollary}
	If $\Delta_{T_{1}}$, $\Delta_{T_{2}}$ are nonzero polynomials, then $\Delta_{\left(\substack{T_{1}\\T_{2}}\right)}$ is also a non-zero polynomial. 
\end{corollary}

\section{Degree bounds}\label{sec:degbound}

More generally, let us denote the smallest degrees of invariants as follows: 
$$
\delta(n_1,\ldots, n_d) := \min \left\{m : \dim \mathrm{Inv}(n_1,\ldots, n_d)_m > 0\right\}
$$
so that $\delta_d(n) = \delta(n, \ldots, n)$ ($d$ times).
Note that $n_i$ divides $\delta(n_1,\ldots, n_d)$ for all $i \in [d]$ and hence   
$$\delta(n_1,\ldots, n_d) \ge \mathrm{lcm}(n_{1},\ldots,n_{d}).$$ 


Below we obtain another lower bound, which sometimes can be larger than the last one. We also show a kind of recursive upper bound.

\begin{theorem} We have the following bounds:

	(i) Let $d \ge 3$ be {odd}. Then
	$$\delta(n_{1},\ldots,n_{d}) \ge \lceil {(n_{1}\ldots n_{d})^{1/(d-1)}} \rceil.$$ 
	
	(ii) Let $\ell \in [2, d - 2]$ and $d > 3$. Then 
	$$\delta_{}(n_1, \ldots, n_d) \le \mathrm{lcm}(\delta_{}(n_1,\ldots, n_\ell), \delta_{}(n_{\ell + 1}, \ldots, n_d)).$$
	(Assuming all $\delta(\cdots) < \infty$.)
\end{theorem}
\begin{proof}
 	(i) Let $m=\delta(n_{1},\ldots,n_{d})$ and $T$ be a balanced table associated with minimal degree $m$ invariant polynomial $\Delta_{T} \neq 0$. Then each column of $T$ is one of the possible $k_{1} \cdots k_{d}$ columns, where $k_i = m/n_i$. Hence if $m > k_{1}\ldots k_{d}$ then the table $T$ has two equal columns and by Corollary~\ref{identicalcol} we get $\Delta_{T} = 0$. Therefore, $$m \le k_{1}\ldots k_{d} = m^{d} (n_{1}\ldots n_{d})^{-1}$$ which implies the inequality. 
	
	(ii) Denote $m_{1} =\delta_{}(n_1,\ldots, n_{\ell}), m_{2} =\delta_{}(n_{\ell + 1}, \ldots, n_d)$ and let $M = \mathrm{lcm}(m_{1}, m_{2})$.
	We construct a nonzero invariant polynomial in $\mathrm{Inv}(n_{1},\ldots,n_{d})$ using invariants from $\mathrm{Inv}(n_{1}, \ldots, n_{\ell})$ and $\mathrm{Inv}(n_{\ell+1}, \ldots, n_d)$. Let $\Delta_{T_{1}} \in \mathrm{Inv}(n_{1}, \ldots, n_{\ell})_{m_{1}}$ and $\Delta_{T_{2}} \in \mathrm{Inv}(n_{\ell+1}, \ldots, n_d)_{m_{2}}$ be fundamental (nonzero) invariants. Let $\tilde T_1 = (T_1 \cdots T_1)$ ($M/m_1$ times) and $\tilde T_2 = (T_2 \cdots T_2)$ ($M/m_2$ times) be $\ell \times M$ and $(d - \ell) \times M$ balanced tables obtained by horizontal concatenations. Let $T = \left(\substack{\tilde T_{1}\\ \tilde T_{2}}\right)$ be $d \times M$ balanced table obtained by vertical concatenation. Let $Y, Z$ be tensors such that $\Delta_{T_1}(Y) \ne 0, \Delta_{T_2}(Z) \ne 0$. By Propositions~\ref{productlemma} and \ref{hconcat}, we have 
	$$
	\Delta_{T}(Y \otimes Z) = \Delta_{\tilde T_1}(Y) \cdot \Delta_{\tilde T_2}(Z) = \Delta_{T_1}(Y)^{M/m_1} \cdot \Delta_{T_2}(Z)^{M/m_2} \ne 0.
	$$
	Hence $\Delta_{T}$ is a nonzero invariant of degree $M$ and we get $\delta(n_1,\ldots, n_d) \le M$.
\end{proof}

\begin{corollary}\label{corlb}
	Let $d \ge 3$ be odd. We have 
	$$\delta_{d}(n) \ge n \lceil {n^{1/(d - 1)}} \rceil.$$ 
\end{corollary}

\begin{corollary}\label{cormon}
	We have: $\delta_{d}(n) \le \delta_{d - 2}(n)$ for $d > 3$. 
\end{corollary}
\begin{proof}
Note that $\delta_2(n) = n$ (whose invariant is the ordinary determinant). Then the upper bound gives:
$$\delta_d(n) \le \mathrm{lcm}(\delta_2(n), \delta_{d-2}(n)) = \mathrm{lcm}(n, \delta_{d-2}(n)) = \delta_{d-2}(n).$$
as needed.
\end{proof}


In particular, $\delta_d(n) \le \delta_3(n)$ for all odd $d \ge 3$. Now we show an upper bound for $\delta_3(n)$.


\begin{lemma}\label{lemmaub}
	We have: $\delta_{3}(n) \le n^{2}$. 
\end{lemma}
\begin{proof}
	Recall that $\dim \mathrm{Inv}_3(n)_{n^{2}} = g(n \times n, n \times n, n \times n) > 0$. Hence there exists an invariant of degree $n^2$ and so $\delta_3(n) \le n^2$.
\end{proof}
We discuss an explicit construction of an invariant of degree $n^2$ in \textsection\,\ref{invn2}; it is conditional on the Alon--Tarsi conjecture for even $n$. 

Note that Theorem~\ref{thm:one}~(i) now follows from Corollary~\ref{corlb}, Corollary~\ref{cormon}, and Lemma~\ref{lemmaub}. 


\section{Dimension sequences}\label{sec:dimseq}
Recall that we denote 
$$
g^{}_d(n, k) := \dim \mathrm{Inv}_d(n)_{kn} = g(\underbrace{n \times k, \ldots, n \times k}_{d \text{ times}}).
$$

\begin{theorem}\label{gthm}
Let $d \ge 3$ be odd. The following properties hold.

(i) Bounds: 
$$
g^{}_d(n, k) \le \binom{k^d}{nk}.
$$
In particular,
$g^{}_d(n, k) = 0  \text{ for } n > k^{d - 1}.$

(ii) Symmetry: 
$$g^{}_d(n, k) = g^{}_d(k^{d - 1} - n, k) \text{ for } n \in [0, k^{d - 1}].$$
In particular, $g_d(0, k) = g_d(k^{d - 1}, k) = 1$.

(iii) Positivity: For fixed $k$, if the sequence $\{g^{}_d(n, k) \}_{n = 0}^{k^{d - 1}}$ is positive for $d= 3$, then it is positive for all odd $d > 3$.
\end{theorem}

\begin{lemma}
Let $d \ge 5$ be odd. The following formula holds:
\begin{align}\label{eq:cdnkron}
g_d(n, k) = \sum_{\mu^{(2)},\ldots\mu^{(d-2)}}  g(\mu^{(1)}, k\times n, \mu^{(2)}) g(\mu^{(2)}, k\times n, \mu^{(3)})\cdots g(\mu^{(d-2)}, k\times n, \mu^{(d-1)}),
\end{align}
where $\mu^{(1)} =k \times n$ and $\mu^{(d-1)} = n\times k$. 
\end{lemma}
\begin{proof}
By conjugation we have 
$[n \times k] \otimes [n \times k] \simeq [k \times n] \otimes [k \times n]$
and hence obtain that  
$$
g_d(n, k) = \text{mult. } [{n \times k}] \text{ in } \underbrace{[{n \times k}] \otimes \cdots \otimes [{n \times k}]}_{(d-1) \text{ times}} = \text{mult. } [{n \times k}] \text{ in } \underbrace{[{k \times n}] \otimes \cdots \otimes [{k \times n}]}_{(d-1) \text{ times}}. 
$$
The formula \eqref{eq:cdnkron} then follows by iteratively decomposing the last tensor product expression.
\end{proof}

For a partition $\lambda$ denote the {\it width} $w(\lambda) := \ell(\lambda') = \lambda_{1}$.

\begin{lemma}[Rectangular bound]\label{cdnrect}
	Let $\mu^{(1)}=k \times n,\mu^{(2)},\ldots,\mu^{(d-2)},\mu^{(d-1)}= n \times k$ be partitions such that  
	$$g(\mu^{(i)},k\times n,\mu^{(i+1)}) > 0 \text{ for } i = 1,\ldots,d-2.$$ 
	Then $\mu^{(i)} \subseteq k^{i} \times k^{d-i}$ for $i = 1,\ldots,d-1$. 
\end{lemma}
\begin{proof}
	We first prove that $\ell(\mu^{(i)}) \le k^{i}$ by induction on $i=1,\ldots, d-2$.  
	For $i=1$ we know that $\ell(k\times n) = k^{1}$. Now assume $\ell(\mu^{(i)}) \le k^{i}$. By Lemma~\ref{kron}(f) the positivity $g(\mu^{(i)},k\times n, \mu^{(i+1)}) > 0$ implies that
	$$\ell(\mu^{(i+1)}) \le \ell(\mu^{(i)}) \cdot \ell(k\times n) \le k^{i+1}$$
	as desired. 
	
	Let us now show that $w(\mu^{(i)}) \le k^{d-i}$ by reverse induction on $i= d-1, \ldots, 1$. For $i = d-1$ we know that $w(n \times k) = k$. Now assume $w(\mu^{(i)}) \le k^{d-i}$. Again by Lemma~\ref{kron}(f) the positivity $g(\mu^{(i)},k\times n, \mu^{(i+1)}) > 0$ implies
	$$w(\mu^{(i-1)}) \le |\mu^{(i)} \cup (k \times n)| \le |(k^{i} \times k^{d-i}) \cup (k \times n)| \le k^{d-i+1}$$
	and the proof follows. 
\end{proof}

\subsection*{Proof of Theorem~\ref{gthm}}
\begin{proof}[(i)]
The Kronecker coefficient $g_{d}(n,k)$ is bounded from above by the number of $d \times nk$ balanced tables with entries from $[k]$ and distinct columns, where the order or columns is not important. 
The maximal number of distinct columns with entries from $[k]$ is at most $k^d$ and hence  the number of such $d \times nk$ balanced tables (as an ordered set of columns) is bounded above by $\binom{k^d}{nk}$. 
\end{proof}

\begin{remark}
Let us also note that this is a quite rough bound; using balanced tables it is also easy to show that 
$$
g_d(n,k) \le \frac{1}{k!^{d-1}} \binom{nk}{n,\ldots, n}^{d-1} = \left(\frac{(nk)!}{n!^k k!} \right)^{d-1}.
$$
\end{remark}
\begin{proof}[(ii)]
	Let $\mu^{(1)}=k\times n, \mu^{(2)},\ldots, \mu^{(d-2)},\mu^{(d-1)}=n\times k$ be partitions representing a nonzero term in the expansion \eqref{eq:cdnkron} of $g_d(n, k)$. 
	By Lemma~\ref{cdnrect} we  have $\mu^{(i)} \subseteq k^{i} \times k^{d-i}$. We are going to map every such $(d-1)$-tuple of partitions $\mu^{(i)}$ to a $(d-1)$-tuple of partitions $\nu^{(i)}$ from the similar expansion of $g_d(k^{d - 1} - n, k)$. 
	
	Let $\nu^{(i)} = k^{i} \times k^{d-i} -_{k^{i}} \mu^{(i)}$ for $i=1,\ldots,d-1$. In particular, $\nu^{(1)} = k \times (k^{d-1} - n)$ and $\nu^{(d-1)}= (k^{d-1} - n) \times k$. 
	In Lemma~\ref{kron}(e) we set $(a,b,c) \to (k^{i},k,k^{d-i-1})$ and apply conjugation to obtain 
	\begin{align*}
	g(\mu^{(i)}, k \times n, &\mu^{(i+1)})  = 
	g((\mu^{(i)})', n \times k, \mu^{(i+1)}) \\
	&= g(
	k^{d-i} \times k^{i} -_{k^{d-i}}(\mu^{(i)})', 
	k^{d-1} \times k -_{k^{d-1}} n \times k, 
	k^{i+1} \times k^{d-i-1} -_{k^{i+1}}\mu^{(i+1)}
	)
	\\&=g(
	k^{i} \times k^{d-i} -_{k^{i}}\mu^{(i)}, 
	k \times (k^{d-1} - n), 
	k^{i+1} \times k^{d-i-1} -_{k^{i+1}}\mu^{(i+1)})
	\\&=g(
	\nu^{(i)}, 
	k \times (k^{d-1} - n), 
	\nu^{(i+1)})
	\end{align*}
	for each $i = 1,\ldots,d-2$. Therefore, nonzero terms in the expansions \eqref{eq:cdnkron} for $g_d(n, k)$ and $g_d(k^{d - 1} - n, k)$ can be bijectively mapped to each other, which implies the symmetry $g_d(n, k) = g_d(k^{d - 1} - n, k)$ for all $n \in [0,k^{d-1}]$.
\end{proof}

\begin{proof}[(iii)]
	We proceed by induction on odd $d \ge 3$. The base case $d = 3$ is the given condition. Assume $d \ge 3$ is odd and $g_d(n, k) > 0$ for all $n \in [0, k^{d-1}]$. 
	
	Let us show that $g_{d+2}(n, k) >  0$ for all $n \in [0, k^{d+1}]$. 
	Fix $n \in [0,k^{d-1}]$ and expand $g_d(n, k)$ as in \eqref{eq:cdnkron}. Since $g_d(n, k) > 0$ there are partitions $\mu^{(1)}=k\times n, \mu^{(2)},\ldots,\mu^{(d-2)},\mu^{(d-1)}=n\times k$ such that
	\begin{align}\label{dtuple}
		g(\mu^{(1)}, k\times n, \mu^{(2)})\,
		g(\mu^{(2)}, k\times n, \mu^{(3)})
		\cdots
		g(\mu^{(d-2)},k\times n, \mu^{(d-1)}) > 0.
	\end{align}
	For each $\ell = 0,\ldots,k^{2}-1$ and $N=\ell k^{d-1} + n$ we will prove the positivity of $g_{d+2}(N, k)$. Note that $N$ ranges in $[\ell k^{d-1}, (\ell+1)k^{d-1}]$ as $n$ ranges in $[0,k^{d-1}]$.
	Set $\nu^{(i)} = k^{i} \times \ell k^{d-i} + \mu^{(i)}$ for $i = 1,\ldots,d-1$. 
	By Lemma~\ref{cdnrect} we have $\mu^{(i)} \subseteq k^{i} \times k^{d-i}$.
	Applying Lemma~\ref{kron}(e) for $\ell$ times with the parameters $(a,b,c)=(k^{i},k,k^{d-i-1})$ we obtain
	\begin{align*}
		g(\mu^{(i)},k\times n,\mu^{(i+1)}) 
		&= 
		g(
			\mu^{(i)} + k^{i} \times \ell k^{d-i}, 
			k \times n + k \times \ell k^{d-1}, 
			\mu^{(i+1)} + k^{i+1} \times \ell k^{d-i-1}
			)\\
		&= g(\nu^{(i)},k\times N,\nu^{(i+1)})  > 0
	\end{align*}
	for all $i = 1,\ldots,d-2$ and $\nu^{(i)} \subseteq  k^{i} \times k^{d+2-i}$. Denote $m = k^{d-1} - n \in [0,k^{d-1}]$.
	Then 
	\begin{align*}
		\nu^{(d-1)} = k^{d-1} \times \ell k + n \times k, 
		\qquad
		(\nu^{(d-1)})' = (\ell+1)k \times n + \ell k \times m.
	\end{align*}
	Let us now define the remaining two partitions for $i = d,d+1$:
	\begin{align*}
		\nu^{(d)} &= ( \ell \times mk + (\ell+1)\times nk )', \\
		\nu^{(d+1)} &= N \times k, 
		\qquad (\nu^{(d+1)})' = k \times (\ell+1)n + k \times \ell m.
	\end{align*}
	Then we claim that the following two coefficients are also positive:
	\begin{align*}
		g(\nu^{(d-1)}, k \times N, \nu^{(d)}) > 0, \quad
		g(\nu^{(d)}, k \times N, \nu^{(d+1)}) > 0,
	\end{align*}
	or by conjugating the arguments in both coefficients we can rewrite them as
	\begin{align}
		g(
			(\ell+1)k \times n + \ell k \times m,\,
			k \times (\ell+1)n + k \times \ell m,\,
	 			(\ell+1) \times kn + \ell \times km) > 0, \label{eq:g1} \\
		g(
			(\ell+1) \times kn + \ell \times km,\, 
			k \times (\ell+1)n + k \times \ell m,\,
			k \times (\ell+1)n + k \times \ell m) > 0 .
			\label{eq:g2}
	\end{align}
	Indeed, in Lemma~\ref{kron}(e) set $(a,b,c) \to (n,k,\ell+1)$ and $(a,b,c) \to (m, k, \ell)$ and apply the semi-group property to the resulting coefficients to obtain the positivity \eqref{eq:g1}. Further, since $\ell \in [0, k^{2}-1]$, by assumption the coefficients
	$g((\ell+1) \times k, k \times (\ell+1),k \times (\ell+1))$
	and $g(\ell \times k, k \times \ell, k \times \ell)$ are positive,
	and hence
	$g(\ell \times km, k \times \ell m , k \times \ell m)$ and 
	$g((\ell+1) \times kn, k \times (\ell+1)n , k \times (\ell+1)n)$ are positive as well by applying the semi-group property to itself $m$ and $n$ times, respectively. Combining the last two coefficients and the semi-group property we obtain the positivity of \eqref{eq:g2}.
	So we showed that $g(\nu^{(i)}, k \times N, \nu^{(i+1)})$ is positive for all $i \in [1,d+1]$, hence the expansion \eqref{eq:cdnkron} of $g_{d+2}(N, k)$ contains a positive term. This completes the induction step. 
\end{proof}

A direct computation (see Appendix~\ref{appa}) shows that $\{g_{3}(n,k)\}_{n = 0}^{k^{2}}$ is positive for $k = 2, 4$. Therefore we obtain the following result.

\begin{corollary}\label{cork24}
For all odd $d \ge 3$ and $k = 2, 4$, the sequence $\{g_{d}(n,k)\}_{n = 0}^{k^{d - 1}}$ is positive.
\end{corollary}

We conjecture that the sequence $\{g_{d}(n,k)\}_{n = 0}^{k^{d - 1}}$ is positive for (a) odd $d \ge 5$ and all $k$, and (b) $d = 3$ and even $k$. Note that for odd $k$ we have $g_3(2, k) = g_3(k^2 - 2, k)= 0$ (e.g. \cite{bvz10}) which seem to be the only exceptional cases; in particular, this gives  
$\delta_3(n) \ge n(k+1)$ 
for odd $n = k^2 - 2$. Furthermore, it seems reasonable to conjecture that for (a), (b) 
the sequence $\{g_{d}(n,k)\}_{n = 0}^{k^{d - 1}}$ (besides being symmetric and positive) is also unimodal. For example, 

\begin{center}
\begin{tabular}{c | c c c c c c c c c c c c c c c c c}
$n$ & 0 & 1 & 2 & 3 & 4 & 5 & 6 & 7 & 8 & 9 & 10 & 11 & 12 & 13 & 14 & 15 & 16 \\
\hline
$g_{3}(n, 4)$ & 1 & 1 & 1 & 2 & 5 & 6 & 13 & 14 & 18 & 14 & 13 & 6 & 5 & 2 & 1 & 1 & 1 \\
$g_{5}(n, 2)$ & 1 & 1 & 5 & 11 & 35 & 52 & 112 & 130 & 166 & 130 & 112 & 52 & 35 & 11 & 5 & 1 & 1
\end{tabular}
\end{center}

\

We provide some more related computational data in Appendix~\ref{appa}. 
  
\begin{lemma}\label{lemmagd}
Let $d \ge 3$ be odd. Assume for fixed $k$ the sequence $\{g^{}_d(n, k) \}_{n = 0}^{k^{d - 1}}$ is positive. Then 

(i) $\delta_{d}(n) = nk$ for all $(k-1)^{d-1} < n \le k^{d-1}$, and 

(ii) $\delta_d(n) \in \{n(k-1), nk \}$ for all $(k-2)^{d-1} < n \le (k-1)^{d-1}$.
\end{lemma}
\begin{proof}
Since $\dim \mathrm{Inv}_d(n)_{nk} = g_{d}(n,k) > 0$ there is an invariant of degree $nk$ for all $n \le k^{d-1}$. Hence, $\delta_d(n) \le nk$. 
(i) On the other hand, by Corollary~\ref{corlb} we have $\delta_{d}(n) \ge nk$ for $(k-1)^{d-1} < n$. Hence, $\delta_{d}(n) = nk$ for $(k-1)^{d-1} < n \le k^{d-1}$.
(ii) By Corollary~\ref{corlb} we have $\delta_{d}(n) \ge n(k-1)$ for $(k-2)^{d-1} < n$.
\end{proof}

\begin{corollary}
Let $d \ge 3$ be odd. We have 
\begin{align*}
\delta_d(n) = n \lceil n^{1/(d - 1)} \rceil\, \text{ for } n \in \{1, \ldots, 2^{d-1}\} \cup \{3^{d-1}, \ldots, 4^{d-1} \} \cup \{k^{d-1} - 1, k^{d-1}  : k \in \mathbb{N}_{\ge 2}\},
\end{align*}
and also 
$
\delta_d(n) \in \{3n, 4n \} 
\, \text{ for } n \in \{2^{d -1} + 1, \ldots, 3^{d-1}\}.
$
\end{corollary}
\begin{proof}
Note that $\delta_d(k^{d-1}) \ge k^d$ and $\delta_d(k^{d-1} - 1) \ge k^d - k$ by Corollary~\ref{corlb}. On the other hand, $\dim \mathrm{Inv}_d(k^{d-1})_{k^d} = g_{d}(k^{d - 1}, k)  = 1$ and $\dim \mathrm{Inv}_d(k^{d-1} - 1)_{k^d - k} = g_{d}(k^{d - 1} - 1, k) = g_d(1, k) = 1$ (the last fact can be easily seen using \eqref{eq:cdnkron} and trivial characters) and so there are invariants of degrees $k^d$ and $k^d - k$. Hence $\delta_d(k^{d-1}) = k^d$ and $\delta_d(k^{d - 1} - 1) = k^d - k$. The rest follows from Corollary~\ref{cork24} and Lemma~\ref{lemmagd}.
\end{proof}

This establishes part of Theorem~\ref{thm:one}~(ii). 

\begin{corollary}\label{cor56}
Let $n > 1$ be fixed. The sequence $\delta_{3}(n) \ge \delta_5(n) \ge \cdots \ge \delta_d(n) \ge \cdots $ stabilizes to $2n$ for $d \ge 1 + \log_2 n$.
\end{corollary}
\begin{proof}
Monotonicity of the sequence is established in Corollary~\ref{cormon}. The stabilization limit follows from the above fact that $\delta_d(n) = 2n$ for $1 < n \le 2^{d - 1}$. 
\end{proof}

\begin{remark}\label{rmkd}
In fact, some more cases for values of $\delta_d(n)$ can be covered by noticing that if $g_{d}(n,k) > 0$ for some $n$ such that $(k-1)^{d -1} < n \le k^{d-1}$ (which is the same as $g_d(k^{d-1} - n, k) > 0$), then $\delta_{d}(n) = nk$. As shown in \cite{LZX21} for $d = 3$ this applies to $n =k^2 - 2$ even, $n = k^2 - 3$, and $n = k^2 - k$.  In \textsection\,\ref{sec:hwv0} we will show that conditionally on a $3$-dimensional analogue of the Alon--Tarsi conjecture for $k$ we get full positivity of $g_{d}(n,k)$ for $n \le k^{d-1}$, and conditionally on the Alon--Tarsi conjecture for $k$ get positivity in the smaller range $k^{d- 1}- k \le n \le k^{d-1}$. Furthermore, unconditionally for even $k$ we will prove positivity of $g_d(n,k)$ in the range $k^{d -1} - \sqrt{k}/2 + 1 \le n \le k^{d - 1}$ (Theorem~\ref{uncon}). 
\end{remark}

\section{Explicit fundamental invariants}\label{sec:fund}
Consider the $d \times k^d$  balanced table $T = (T_{i j})$ given by 
	$$
	T_{i j} = \floor*{\frac{j-1}{k^{i}}} 
	\pmod k + 1.
	$$
and let $\widetilde T$ be the $d \times (k^d - k)$ balanced table obtained by removing from $T$ the columns $(1,\ldots, 1)^T, \ldots, (k, \ldots, k)^T$.

For example, for $d = 3$ and $k = 2$ we have
$$
T = 
\begin{pmatrix}
1 1 1 1 2 2 2 2\\
1 1 2 2 1 1 2 2\\
1 2 1 2 1 2 1 2
\end{pmatrix}, \qquad 
\widetilde T = 
\begin{pmatrix}
1 1 1 2 2 2 \\
1 2 2 1 1 2 \\
2 1 2 1 2 1 
\end{pmatrix}.
$$

The table $T$ has the property that in every row each element in $[k]$ appears $k^{d-1}$ times and its $k^d$ columns are all distinct, which is a unique balanced table (up to a permutation of columns) with such a property. In addition, the columns of $T$ are ordered lexicographically.

Let us denote 
\begin{align*}
F_{d,k} := \Delta_T, \qquad \widetilde F_{d,k} := \Delta_{\widetilde T}.
\end{align*}
the corresponding  invariants. 

\begin{theorem}\label{nary}
	Let $d \ge 3$ be odd. 
	Then 
	
	(i) $F_{d,k} \in \mathrm{Inv}_d(k^{d-1})_{k^d}$ is the unique (up scale) fundamental invariant of degree $\delta_d(k^{d-1}) = k^d$.
	
	(ii) $\widetilde F_{d,k} \in \mathrm{Inv}_d(k^{d-1} - 1)_{k^d - k}$ is the unique (up scale) fundamental invariant of degree $\delta_d(k^{d-1} - 1) = k^d - k$.
\end{theorem}
\begin{proof}
The existence and uniquencess of such fundamental invariants follows from the fact that $\dim \mathrm{Inv}_d(k^{d-1})_{k^d} = g_{d}(k^{d-1}, k) = 1$ and $\dim \mathrm{Inv}_d(k^{d-1} - 1)_{k^d - k} = g_{d}(k^{d-1} - 1, k) = g_d(1, k)= 1$ shown in the previous section. (i) This invariant written as $\Delta_T$ corresponds to a $d \times k^d$ balanced table with distinct columns.  Note that there is a unique such table when we disregard the order of columns. 
(ii) A nonzero $\Delta$ polynomial from $\mathrm{Inv}_d(k^{d-1} - 1)_{k^d - k}$ must be indexed by a $d \times (k^d - k)$ table obtained from the full table $T$ by removing $k$ columns which leave the resulting table balanced, i.e. in these columns every entry from $[k]$ appears once in every row. Note that if in any row of $T$  we switch the values of $i$ and $j$ from $[k]$, the resulting $\Delta$ polynomial  will not change (by definition, as it corresponds to $\Delta_S$ with the same tuple of set partitions $S$).   Hence by performing these swap operations we may assume that the $k$ columns removed from $T$ are exactly $(1, \ldots, 1)^T, \ldots, (k, \ldots, k)^T$, and the resulting nonzero polynomial is $\Delta_{\widetilde T}$.
\end{proof}

\begin{remark}
The fundamental invariant $F_{3, k}$ was introduced in \cite{bi}, and $\tilde F_{3,k}$ (formulated via obstruction designs) was introduced in \cite{LZX21}.
\end{remark}

\section{Latin hypercubes}\label{sec:latin}

\subsection{An invariant of degree $k^2$}\label{invn2}
A \textit{Latin square} of length $k$ is an $k \times k$ matrix whose every row and column is a permutation of $[k]$. Let $\mathcal{C}_{2}(k)$ be the set of Latin squares of length $k$. The {\it sign} of a Latin square $L$ is the product of signs of permutations in rows and columns, denoted by $\sgn(L)$. The {\it Alon--Tarsi number} $AT(k)$ is the following signed sum over Latin squares:
$$
    AT(k) := \sum_{L \in \mathcal{C}_{2}(k)} \sgn(L).
$$
It is not difficult to show that $AT(k) = 0$ for odd $k$. The {Alon--Tarsi conjecture} \cite{at} states that $AT(k) \neq 0$ for even $k$. 

We denote by $I_k := (\delta_{i_1 i_2} \cdots \delta_{i_1 i_d})_{i_1,\ldots, i_d \in [k]}$ the unit tensor. 
\begin{proposition}\label{ATtable}
	Let $d \ge 3$ be odd, $n$ be even and $T$ be $d \times k^{2}$ balanced table given by
	\begin{align*}
		T = \begin{pmatrix}
			1^{k}\, 2^{k} \cdots k^{k}\\
			\cdots \\
			1^{k}\, 2^{k} \cdots k^{k}\\
			e_{k}\, e_{k} \cdots e_{k}
		\end{pmatrix}
	\end{align*}
	where we denote $j^{k} = \underbrace{j j \ldots j}_{k\textit{ times}}$ and $e_{k} = 123\ldots k$.
	Then $\Delta_{T}(I_{k}) \neq 0$ is equivalent to the Alon-Tarsi conjecture $AT(k) \ne 0$. In particular, if $AT(k) \ne 0$ then $\Delta_T \in \mathrm{Inv}_d(k)_{k^2}$ is a nonzero invariant of degree $k^2$. 
\end{proposition}
\begin{proof}
	Consider the evaluation of $\Delta_{T}$ at identity tensor. Each nonzero term in the sum \eqref{hdetdef} corresponds to a map $\sigma : [k^2] \to [k]$ such that 
	$(\sigma({1+(i-1)k}),\ldots,\sigma({ik})) \in S_{k}$ for each $i \in [k]$ (from the first $d-1$ rows of $T$), and also $(\sigma({i}), \sigma({i+k}), \ldots, \sigma({i+(k-1)k}) ) \in S_{k}$ for each $i \in [k]$ (from the last row of $T$). Then let $C = (C_{i j})$ be given by $C_{i j} = \sigma({i + (j-1)k})$. The resulting matrix is indeed a Latin square, since each column and row is a permutation. Let $r_1,\ldots, r_k$ and $c_1,\ldots, c_k$ be rows and columns of $C$, respectively; each of them is a permutation of $S_k$. We then have 
	$$
	\sgn(\sigma) = \sgn(r_1 \cdots r_k)^{d-1}\, \sgn(c_1 \cdots c_k) = \sgn(c_1 \cdots c_k)
	$$
	and so the sign corresponding to $\sigma$ in the sum is exactly the column sign of $C$, i.e. the product of signs of columns. Hence, $\Delta_{T}(I_k)$ is the sum of these signs which is known to be equivalent to the Alon--Tarsi conjecture $AT(k) \ne 0$. 
\end{proof}
\begin{remark}
	The conjecture about difference of column even and column odd Latin squares is also known as the Huang--Rota conjecture \cite{hr}, which is equivalent to the Alon--Tarsi conjecture.
\end{remark}

\begin{remark}
The Alon--Tarsi conjecture $AT(k)\ne 0$ is known to hold for $k = p \pm 1$ where $p \ge 3$ is a prime \cite{dri1, glynn}.
\end{remark}

\subsection{Magic sets, Latin cubes and signs}
	We refer to elements of the box $[k]^{d}$ as {\it cells}. 
	A {\it slice} of $[k]^d$ is a subset of all cells with fixed $\ell$-th coordinate (called {\it direction}) for some $\ell \in [d]$.
	A {\it diagonal} of $[k]^d$ is a subset of size $k$ with no two cells lying in the same slice. 
	
	A \textit{magic set} is a subset of $[k]^d$ which has an equal number of elements in every slice of $[k]^d$. We can represent a magic set $T$ as a {\it magic hypermatrix} with $1$ at cells corresponding to elements of $T$ and $0$ elsewhere, 
	which is a natural generalization of {\it magic squares}. 

\

	A {\it $d$-dimensional partial Latin hypercube} of length $k$ and cardinality $M$ is a 
	function $C : [k]^d \to \{0, 1,\ldots,n={M}/{k} \}$  satisfying the following two conditions:
  \begin{itemize}
	\item the cells $\{a \in [k]^d : C(a) \ne 0 \}$ form a magic set of cardinality $M$, and 
	\item for each slice $A$ of $[k]^d$, the nonzero values $C(a) \ne 0$ written in lexicographical order of $a \in A$,  form a permutation of $[n]$ which we call {\it $C$-permutation} of this slice. 
  \end{itemize}	
  The underlying magic set of $C$ is called its \textit{type}. 
  The set of $d$-dimensional partial Latin hypercubes of type $T$ and length $k$ is denoted by $\mathcal{C}_{d}(k,T)$. If $T = [k]^d$, then $M = nk = k^{d}$ and we call the elements of $\mathcal{C}_{d}(k,T)$ as (full) Latin hypercubes. The set of Latin hypercubes of length $k$ is denoted by $\mathcal{C}_{d}(k) = \mathcal{C}_{d}(k, [k]^d)$. See Fig.~\ref{fig1} with some examples.

\begin{definition}
  Let $C \in \mathcal{C}_{d}(k,T)$ be a partial Latin hypercube of length $k$ and type $T \subseteq [k^d]$ with $|T| = nk$. Define the following sign functions:
  \begin{itemize}
  	\item \textit{the directional sign} $\sgn_{\ell}(C)$ is the product of signs of $C$-permutations of the slices in the direction $\ell \in [d]$; 
	\item \textit{the full sign} $\sgn(C) := \sgn_{1}(C)\cdots\sgn_{d}(C)$;
	\item \textit{the symbol sign} $\mathrm{ssgn}(C)$ is the product of {\it subsigns} of each $i \in [n]$ defined as follows. Let $\mathrm{diag}_{i}(C) := \{a \in [k]^d : C(a) = i \}$. Since $i \in [n]$ is present in each slice exactly once, then $\mathrm{diag}_{i}(C)$ is a diagonal and hence can be described by permutations $\pi_2, \ldots, \pi_d$ of $[k]$ so that $\mathrm{diag}_{i}(C) = \{(j,\pi_{2}(j),\ldots,\pi_{d}(j)) : j \in [k] \}$. Then the subsign of $i$ is $\sgn(\pi_{2})\cdots\sgn(\pi_{d})$.
  \end{itemize}  
\end{definition}


\begin{figure}
	\newcounter{l}
	\newcounter{L}
	\setcounter{l}{1}
	\setcounter{L}{\value{l}*3}
	\newcommand{\drawsquare}[4] {
		\draw[-, #4] (#1,#2+\value{L},#3) -- (#1,#2,#3) -- (#1,#2,#3+\value{L});
		\foreach \y in {\value{l},2*\value{l},3*\value{l}} {
			\foreach \z in {\value{l},2*\value{l},3*\value{l}} {
				\draw[-, #4] (#1,\y-\value{l},\z) -- (#1,\y,\z) -- (#1,\y,\z-\value{l});
			}
		}
	}
	\newcommand{\fillcell}[4] {
		\draw[#4] (#1,#2,#3) -- (#1,#2+\value{l},#3) -- (#1,#2+\value{l},#3+\value{l}) -- (#1,#2,#3+\value{l}) -- cycle;
	}
	\newcommand{\putcell}[5] {
		\filldraw[#4] (#1, #2+\value{l}/2, #3+\value{l}/2) node{#5};
	}
	\begin{center}
		\centering
          \begin{minipage}{0.3\textwidth}
			\centering
			\resizebox{\textwidth}{!}{
				\begin{tikzpicture}[x=1.7cm,y=1cm,z=0.5cm,rotate around y=146,rotate around x=0]
				\draw[->, thin] (3,4,0) -- (2.3,4,0)  node[above] {$i_1$};
					\draw[->, thin] (3,4,0) -- (3,2.7,0) node[left] {$i_2$};
					\draw[->, thin] (3,4,0) -- (3,4,1.5) node[above] {$i_3$};
					\drawsquare{0}{0}{0}{color=black, thick};
					\putcell{0*\value{l}}{0*\value{l}}{0*\value{l}}{black}{2};
					\putcell{0*\value{l}}{0*\value{l}}{1*\value{l}}{black}{3};
					\putcell{0*\value{l}}{1*\value{l}}{0*\value{l}}{black}{1};
					
					\drawsquare{\value{l}}{0}{0}{color=black, thick};
					\putcell{1*\value{l}}{1*\value{l}}{1*\value{l}}{black}{2};
					\putcell{1*\value{l}}{1*\value{l}}{2*\value{l}}{black}{3};
					\putcell{1*\value{l}}{2*\value{l}}{1*\value{l}}{black}{1};
					
					\drawsquare{2*\value{l}}{0}{0}{color=black, thick};
					\putcell{2*\value{l}}{0*\value{l}}{2*\value{l}}{black}{1};
					\putcell{2*\value{l}}{2*\value{l}}{0*\value{l}}{black}{3};
					\putcell{2*\value{l}}{2*\value{l}}{2*\value{l}}{black}{2};		
				\end{tikzpicture}
			}
			\subcaption{\scriptsize A partial Latin cube $C(i_1, i_2, i_3)$ of length $k = 3$ and cardinality $M = 9$; its type (cells of nonzero values) is a magic set. } 
		\end{minipage}
		\qquad\qquad
          \begin{minipage}{0.25\textwidth}
			\centering
			\resizebox{\textwidth}{!}{
	
			\begin{tikzpicture}[x=1.7cm,y=1cm,z=0.5cm,rotate around y=146,rotate around x=-0]
				\drawsquare{0}{0}{0}{color=black, thick};
				\putcell{0*\value{l}}{0*\value{l}}{0*\value{l}}{black}{2};
				\putcell{0*\value{l}}{0*\value{l}}{1*\value{l}}{black}{3};
				\putcell{0*\value{l}}{0*\value{l}}{2*\value{l}}{black}{6};
				
				\putcell{0*\value{l}}{1*\value{l}}{0*\value{l}}{black}{1};
				\putcell{0*\value{l}}{1*\value{l}}{1*\value{l}}{black}{9};
				\putcell{0*\value{l}}{1*\value{l}}{2*\value{l}}{black}{4};
				
				\putcell{0*\value{l}}{2*\value{l}}{0*\value{l}}{black}{7};
				\putcell{0*\value{l}}{2*\value{l}}{1*\value{l}}{black}{5};
				\putcell{0*\value{l}}{2*\value{l}}{2*\value{l}}{black}{8};

				\drawsquare{\value{l}}{0}{0}{color=black, thick};
				\putcell{1*\value{l}}{0*\value{l}}{0*\value{l}}{black}{4};
				\putcell{1*\value{l}}{0*\value{l}}{1*\value{l}}{black}{7};
				\putcell{1*\value{l}}{0*\value{l}}{2*\value{l}}{black}{5};
				
				\putcell{1*\value{l}}{1*\value{l}}{0*\value{l}}{black}{8};
				\putcell{1*\value{l}}{1*\value{l}}{1*\value{l}}{black}{2};
				\putcell{1*\value{l}}{1*\value{l}}{2*\value{l}}{black}{3};
				
				\putcell{1*\value{l}}{2*\value{l}}{0*\value{l}}{black}{6};
				\putcell{1*\value{l}}{2*\value{l}}{1*\value{l}}{black}{1};
				\putcell{1*\value{l}}{2*\value{l}}{2*\value{l}}{black}{9};

				\drawsquare{2*\value{l}}{0}{0}{color=black, thick};
				\putcell{2*\value{l}}{0*\value{l}}{0*\value{l}}{black}{9};
				\putcell{2*\value{l}}{0*\value{l}}{1*\value{l}}{black}{8};
				\putcell{2*\value{l}}{0*\value{l}}{2*\value{l}}{black}{1};
				
				\putcell{2*\value{l}}{1*\value{l}}{0*\value{l}}{black}{5};
				\putcell{2*\value{l}}{1*\value{l}}{1*\value{l}}{black}{6};
				\putcell{2*\value{l}}{1*\value{l}}{2*\value{l}}{black}{7};
				
				\putcell{2*\value{l}}{2*\value{l}}{0*\value{l}}{black}{3};
				\putcell{2*\value{l}}{2*\value{l}}{1*\value{l}}{black}{4};
				\putcell{2*\value{l}}{2*\value{l}}{2*\value{l}}{black}{2};
			\end{tikzpicture}
			}
			\subcaption{\scriptsize A Latin cube in $\mathcal{C}_3(3)$.
			} 
		\end{minipage}
		\qquad 		
		\begin{minipage}{0.23\textwidth}
			\centering
			\resizebox{\textwidth}{!}{
				\begin{tikzpicture}[x=1.7cm,y=1cm,z=0.5cm,rotate around y=146,rotate around x=-0]
					\drawsquare{0}{0}{0}{color=black, thick};
					\putcell{0*\value{l}}{0*\value{l}}{0*\value{l}}{black}{$\bullet$};
					\putcell{0*\value{l}}{0*\value{l}}{1*\value{l}}{black}{$\bullet$};
					\putcell{0*\value{l}}{1*\value{l}}{0*\value{l}}{black}{$\bullet$};
					
					\drawsquare{\value{l}}{0}{0}{color=black, thick};
					\putcell{1*\value{l}}{0*\value{l}}{0*\value{l}}{black}{$\bullet$};
					\putcell{1*\value{l}}{1*\value{l}}{2*\value{l}}{black}{$\bullet$};
					\putcell{1*\value{l}}{2*\value{l}}{1*\value{l}}{black}{$\bullet$};
					
					\drawsquare{2*\value{l}}{0}{0}{color=black, thick};
					\putcell{2*\value{l}}{1*\value{l}}{2*\value{l}}{black}{$\bullet$};
					\putcell{2*\value{l}}{2*\value{l}}{1*\value{l}}{black}{$\bullet$};
					\putcell{2*\value{l}}{2*\value{l}}{2*\value{l}}{black}{$\bullet$};
				\end{tikzpicture}
			}
			\subcaption{\scriptsize A magic set which does not give rise to a partial Latin cube. } 
		\end{minipage}
	\end{center}
	\caption{Examples of Latin cubes and magic sets presented by slices in direction $1$.}\label{fig1}
\end{figure}
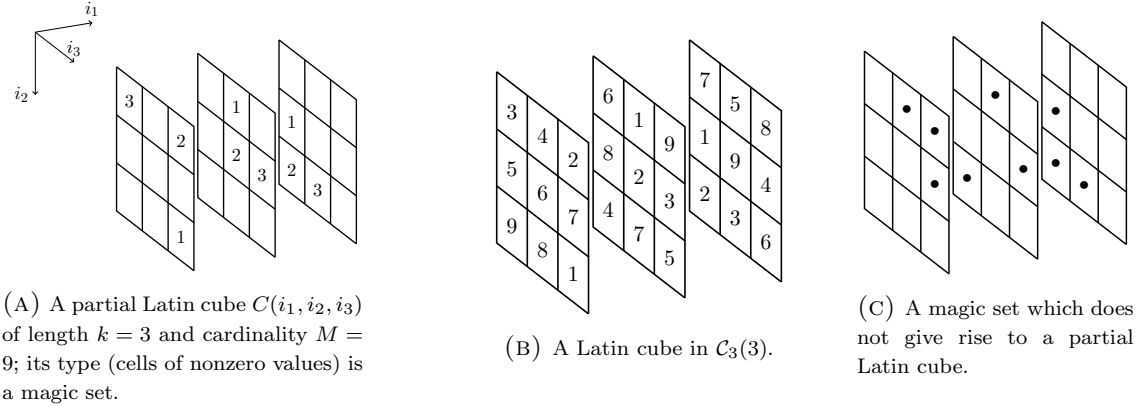

\begin{example}
Let us compute the signs of the partial Latin cube $C \in \mathcal{C}_{3}(3, T)$ shown on Fig.~\ref{fig1}~(A) 
which is given by 
\begin{align*}
C(1,3,3) &= C(2,1,2) = C(3,2,1) = 1, \\
C(1,1,3) &= C(2,2,2) = C(3,3,1) = 2, \\
C(1,1,1) &= C(2,2,3) = C(3,3,2) = 3,
\end{align*}
and $C(i_1, i_2, i_3) = 0$ otherwise.
Its slices in: 
\begin{itemize}
\item[] direction $1$ have 
$C$-permutations 
$(3,2,1)$, $(1,2,3)$, $(1,2,3)$ and $\sgn_1(C) = -1$, 

\item[] direction $2$ have 
$C$-permutations 
$(3,2,1)$, $(2,3,1)$, $(1,2,3)$ and $\sgn_2(C) = -1$, 

\item[] direction $3$ have 
$C$-permutations 
$(3,1,2)$, $(1,2,3)$, $(2,1,3)$ and $\sgn_3(C) = 1$,
\end{itemize}
and hence, $\sgn(C) = 1$. 
To compute the symbol sign, we have: 
\begin{itemize}
\item[] $\mathrm{diag}_1(C) = \{(1,3,3), (2, 1, 2),  (3,2,1)\}$, subsign of $1$ is $\sgn(3,1,2)\, \sgn(3,2,1) = -1$,

\item[] $\mathrm{diag}_2(C) = \{(1,1,3), (2, 2, 2),  (3,3,1)\}$, subsign of $2$ is $\sgn(1,2,3)\, \sgn(3,2,1) = -1$,

\item[] $\mathrm{diag}_3(C) = \{(1,1,1), (2, 2, 3),  (3,3,2)\}$, subsign of $3$ is $\sgn(1,2,3)\, \sgn(1,3,2) = -1$, 
\end{itemize}
and hence, $\mathrm{ssgn}(C) = -1$.
\end{example}

\begin{example}
Let us compute the signs of the Latin cube $C \in \mathcal{C}_3(3)$ shown in Fig.~\ref{fig1}~(B) which is presented by slices in direction $1$ (as matrices) given by 
\begin{align*}
C(1, \cdot, \cdot) 
= \begin{pmatrix}
3 & 4  & 2\\
5 & 6 & 7\\
9 & 8 & 1
\end{pmatrix}, 
\quad
C(2, \cdot, \cdot) 
= \begin{pmatrix}
6 & 1  & 9\\
8 & 2 & 3\\
4 & 7 & 5
\end{pmatrix}, 
\quad 
C(3, \cdot, \cdot) 
= \begin{pmatrix}
7 & 5  & 8\\
1 & 9 & 4\\
2 & 3 & 6
\end{pmatrix}.
\end{align*}
Then we have
\begin{align*}
\sgn_1(C) &= \sgn(3,4,2,5,6,7,9,8,1)\, \sgn(6,1,9,8,2,3, 4,7,5)\, \sgn(7,5,8,1,9,4,2,3,6) = -1,\\ 
\sgn_2(C) &= \sgn(3,4,2,6,1,9,7,5,8)\, \sgn(5,6,7,8,2,3,1,9,4)\, \sgn(9,8,1,4,7,5,2,3,6) = -1, \\ 
\sgn_3(C) &= \sgn(3,5,9,6,8,4,7,1,2)\, \sgn(4,6,8,1,2,7,5,9,3)\, \sgn(2,7,1,9,3,5,8,4,6) = 1, 
\end{align*}
and hence, $\sgn(C) = 1$. To compute the symbol sign we have:
\begin{itemize}
\item[] $\mathrm{diag}_1(C) = \{(1,3,3), (2, 1, 2),  (3,2,1)\}$, subsign of $1$ is $\sgn(3,1,2)\, \sgn(3,2,1) = -1$,

\item[] $\mathrm{diag}_2(C) = \{(1,1,3), (2, 2, 2),  (3,3,1)\}$, subsign of $2$ is $\sgn(1,2,3)\, \sgn(3,2,1) = -1$,

\item[] $\mathrm{diag}_3(C) = \{(1,1,1), (2, 2, 3),  (3,3,2)\}$, subsign of $3$ is $\sgn(1,2,3)\, \sgn(1,3,2) = -1$, 

\item[] $\mathrm{diag}_4(C) = \{(1,1,2), (2, 3, 1),  (3,2,3)\}$, subsign of $4$ is $\sgn(1,3,2)\, \sgn(2,1,3) = 1$,

\item[] $\mathrm{diag}_5(C) = \{(1,2,1), (2, 3, 3),  (3,1,2)\}$, subsign of $5$ is $\sgn(2,3,1)\, \sgn(1,3,2) = 1$,

\item[] $\mathrm{diag}_6(C) = \{(1,2,2), (2, 1, 1),  (3,3,3)\}$, subsign of $6$ is $\sgn(2,1,3)\, \sgn(2,1,3) = 1$, 

\item[] $\mathrm{diag}_7(C) = \{(1,2,3), (2, 3, 2),  (3,1,1)\}$, subsign of $7$ is $\sgn(2,3,1)\, \sgn(3,2,1) = 1$,

\item[] $\mathrm{diag}_8(C) = \{(1,3,2), (2, 2, 1),  (3,1,3)\}$, subsign of $8$ is $\sgn(3,2,1)\, \sgn(2,1,3) = 1$,

\item[] $\mathrm{diag}_9(C) = \{(1,3,1), (2, 1, 3),  (3,2,2)\}$, subsign of $9$ is $\sgn(3,1,2)\, \sgn(1,3,2) = -1$, 
\end{itemize}
and hence $\mathrm{ssgn}(C) = 1$. Note that 
$\sgn_1(C)^2\, \sgn_2(C)\, \sgn_3(C) \, \mathrm{ssgn}(C)  = -1$ 
which is as in Corollary~\ref{Latinsigns} below.
\end{example}

\begin{example}
It can also be noted that not every magic set gives rise to a Latin hypercube. For example, the set 
$$
T = \{(1,1,2), (1,1,3), (1,2,3), (2,1,2), (2,2,3), (2,3,1), (3,2,1), (3,3,1), (3,3,2) \}
$$
shown in Fig.~\ref{fig1}~(C) is a magic set of length $k = 3$ and cardinality $M = 9$ but it cannot be decomposed as a disjoint union of diagonals, and hence there exist no partial Latin cubes of type $T$.
Note that for $d = 2$ every magic square can be decomposed as a union of diagonals (transversals), which is (a special case of) Birkhoff's theorem.
\end{example}

\subsection{$d$-dimensional Alon--Tarsi numbers} Let us denote
\begin{align*}
	AT_d(k, T) := \sum_{C \in \mathcal{C}_d(k, T)} \sgn(C), \qquad AT_d(k) := AT_d(k, [k]^d).
\end{align*}
Note that $AT_{2}(k) = AT(k)$ is the Alon--Tarsi number.
\begin{proposition}\label{Latinidentitygen}
	Let $\Delta_{T} \in \mathrm{Inv}_d(n)_{nk}$, where $T$ is a $d\times nk$ balanced table with distinct columns ordered lexicographically. Then
	\begin{align}
		\Delta_{T}(I_{n}) 
		= AT_{d}(k,T),
	\end{align}
	where in r.h.s. $T$ is viewed as a set of its columns. 
\end{proposition}
\begin{proof}
 	Let us show a bijection between nonzero terms of $\Delta_{T}(I_n)$ and partial Latin hypercubes. Let $\sigma : [nk] \to [n]$ be a map corresponding to a nonzero term in the expansion \eqref{hdetdef} of $\Delta_{T}(I_n)$ (note that in this expansion we must have $\sigma_1 = \cdots = \sigma_d = \sigma$ as we evaluate at the unit tensor).  Let $T_{j} \in [k]^d$ be the $j$-th column of $T$ for $j \in [nk]$. 
	We construct a partial Latin hypercube $C_{\sigma}$ by setting $C_{\sigma}(T_j) = \sigma({j})$. (At all remaining cells $C$ is $0$.) One can see that the resulting function $C_{\sigma} : [k]^d \to \{0, \ldots, n\}$ is indeed a partial Latin hypercube of length $k$ and type $T$, i.e. $C_{\sigma} \in \mathcal{C}_d(k, T)$. 
	Since the columns of $T$ are ordered lexicographically, we also have $\sgn_T(\sigma) = \sgn(C_{\sigma})$. Conversely, given $C \in \mathcal{C}_d(k, T)$ we set $\sigma(j) = C(T_j)$.
	This gives a sign preserving bijection. 
\end{proof}

\begin{corollary}\label{cor:fLatin}
	For the fundamental invariant $F_{d,k}$ we have
	\begin{align}
		F_{d,k}(I_{k^{d-1}}) = \sum_{C \in \mathcal{C}_{d}(k)} \sgn(C) = AT_{d}(k).
	\end{align}
\end{corollary}

The following statement is a generalization of Proposition~\ref{ATtable}.
\begin{proposition}\label{ATtable1}
	Let $d \ge 3$ odd, $k$ be even, and $\ell \le d$ be odd, $n=k^{\ell-1}$. Let $T$ be a $d\times k^{\ell}$ balanced table 
	whose row $i \in [d]$ denoted as $T_{i}$ is given by
	\begin{align*}
		T_{i} =  
			\begin{cases}
			    (1^{k^{\ell-i}} \ldots k^{k^{\ell-i}})^{k^{i-1}}  & \text{if } i < \ell, \\
			    (e_{k})^{k^{\ell-1}} & \text{otherwise},
			\end{cases}
	\end{align*}
	where $j^{n} = \underbrace{j j\ldots j}_{n\textit{ times}}$ and $e_{k} = 123\ldots k$. Then $\Delta_{T}(I_{n}) = AT_{\ell}(k)$. In particular, if $AT_{\ell}(k) \ne 0$ then $\Delta_T$ is a nonzero invariant of degree $k^{\ell}$.
\end{proposition}
\begin{proof}
	Let $\sigma=(\sigma_{1},\ldots,\sigma_{d})$ be a $d$-tuple of maps $[k^{\ell}] \to [k]$ corresponding to a nonzero term in the expansion \eqref{hdetdef} of $\Delta_{T}(I_{n})$. Note that we must have $\sigma_{\ell}=\ldots=\sigma_{d}$. Let $H$ be a balanced table formed with the first $\ell$ rows of $T$ and $\sigma'=(\sigma_{1},\ldots,\sigma_{\ell})$, and $H'$ be the table of the remaining $d - \ell$ rows. Note that $H$ is a table of all possible $[k]^{\ell}$ columns and hence $\Delta_H = F_{\ell, k}$. Then we have  
	$$\sgn_{T}(\sigma) = \sgn_{H}(\sigma')\, \sgn_{H'_{}}(\sigma_{\ell})^{d-\ell} = \sgn_{H}(\sigma').$$ 
	Hence there is a  sign-preserving bijection  between nonzero terms of the sum in the expansion of $\Delta_{T}(I_{n})$ and nonzero terms in the expansion of $\Delta_{H}(I_{n})$. By Corollary~\ref{cor:fLatin} we have $\Delta_{T}(I_{n}) = \Delta_{H}(I_{n}) = AT_{\ell}(k)$.
\end{proof}

\begin{lemma}\label{sgnchange}
	Let $C$ be a $d$-dimensional partial Latin hypercube of length $k$ and cardinality $nk$. Let $C'$ be a partial Latin hypercube obtained from $C$ by exchanging its values $i\neq j$ from $[n]$. Then $$\sgn(C') = (-1)^{dk}\sgn(C).$$
\end{lemma}
\begin{proof}
	The exchange results in a single transposition in each slice, and since there are $dk$ slices, the proposition follows.
\end{proof}

\begin{corollary}\label{cor:zeroev}
	Let $d$ and $k$ be odd and $n > 1$. For any $\Delta_{T} \in \mathrm{Inv}_d(n)_{kn}$ we have 
	$
	\Delta_{T}(I_{n}) = 0.
	$
	In particular, for any invariant $F \in \mathrm{Inv}_d(n)_{kn}$ we  have $F(I_n) = 0$.
\end{corollary}
\begin{proof}
	By Lemma~\ref{sgnchange} the exchange of two values (say $1$ and $2$) defines the sign reversing involution on $\mathcal{C}_{d}({k,T})$, and hence all terms cancel out. 
\end{proof}

\begin{remark}\label{remstab}
On the other hand, it is known that there is an invariant $\Delta_T$ such that $\Delta_T(I_n) \ne 0$ since the unit tensor $I_n$ is {\it semistable} for the action of $G$, cf. \cite{widg}. Hence its degree must be $kn$ for some even $k \ge \delta_d(n)/n$. 
\end{remark}

In particular, we have $AT_d(k) = 0$ for odd $k$. Let us show that for $k = 2$ it is nonzero.

\begin{proposition}\label{at2nz}
	For all $d \ge 2$, we have $AT_{d}(2) > 0$. 
\end{proposition}
\begin{proof}
	Let $C \in \mathcal{C}_{d}({2})$. Note that for each $a \in [2]^{d}$ we must have $C(a) = C(\bar a)$, where $\bar a := (3,\ldots,3) - a$. Thus any two parallel slices contain permutations reverse to each other, namely, if $\pi_{1}$ is a permutation (collected lexicographically) in the slice $1$ in the direction $i \in [d]$ then the permutation $\pi_{2}$ from the unique parallel slice $2$ (also collected lexicographically) satisfies $\pi_{2}= w_{0} \circ \pi_{1}$, where $w_{0}$ is the reverse of the identity permutation of length $2^{d-1}$. Hence, 
	$$\sgn_{i}(C) = \sgn(\pi_{1})\, \sgn(\pi_{2}) = \sgn(w_{0}), $$  
	$$\sgn(C) = \sgn_{1}(C)\cdots\sgn_{d}(C) = \sgn(w_{0})^{d} = (-1)^{2^{d-2} d} = 1$$ for all $d\ge 2$ and the result follows.
\end{proof}
\begin{remark}
In \cite{bi} it is also noted that computations show $AT_3(4) \ne 0$.
\end{remark}

\subsection{Relation between signs} 
For a sequence $a = (a_1,\ldots,a_n) \in \mathbb{Z}^n$ let 
$$\mathrm{inv}(a) := |\{ (i,j): a_i > a_j,\, 1 \le i<j\le n \}|, \qquad \mathrm{msgn}(a) := (-1)^{\mathrm{inv}(a)}.$$ 
be the number of inversions and the {\it multi-sign} of $a$. 
When $a \in S_n$ we have $\mathrm{msgn}(a) = \sgn(a)$.
  
Let $T=\left\{x_1, \ldots, x_{nk}\right\} \subseteq [k]^d$ be a magic set with the cells $x_1 <\ldots < x_{nk} \in [k]^d$ ordered lexicographically. Let $x_i = ( x^{(1)}_i, \ldots, x^{(d)}_i )$ for $i \in [nk]$ and denote $T_\ell = ( x^{(\ell)}_1,\ldots,x^{(\ell)}_{nk} )$ for $\ell \in [d]$. 
Then we also define the sign of a magic set as follows: 
$$
	\sgn(T) := \mathrm{msgn}(T_1)\cdots\mathrm{msgn}(T_d).
$$

Now we show that the signs of Latin hypercubes defined earlier are related as follows.  This result will be useful in the next section.

\begin{theorem}\label{pLatinsigns}
	Let $d > 2$ and $T \subseteq [k]^d$ be a magic set. For every partial Latin hypercube $C \in \mathcal{C}_d(k,T)$ we have
	$$
		\sgn_{1}(C)^{d-1} \sgn_{2}(C) \cdots \sgn_{d}(C)\, \mathrm{ssgn}(C) = \sgn(T).
	$$
\end{theorem}
\begin{proof}
In this proof we shall consider all sums and equalities mod 2. 
We use the notation 
$$[A] := 
\begin{cases} 
1, & {A = true},\\
0, & {A = false}.
\end{cases}
$$ 
We express signs as $\sgn(\sigma) = (-1)^{\text{\# inversions of } \sigma}$ for a permutation $\sigma$. 

For all $a = (a_1,\ldots, a_d), b = (b_1,\ldots, b_d) \in [k]^d$,
define the functions $A_{\ell}: (a,b) \mapsto \{0,1\}$ for $\ell \in [d]$  and $B:(a,b) \mapsto \{0,1\}$ given by
	\begin{align*}
		A_{\ell}(a, b) & := [a_{\ell} = b_{\ell}]\, [a < b]\, [C(a) > C(b)], \\
		B(a,b) & := [a_{1} < b_{1}]\, [C(a) = C(b)]\sum_{i=2}^{d}~  [a_{i} > b_{i}].
	\end{align*}
	where the relation $a < b$ corresponds to lexicographical order. Note that $A_{\ell}(a,b)$ expresses if the cells $a$ and $b$  lie in the same slice in the $\ell$-th direction and form an inversion of a $C$-permutation. Similarly, $B$ indicates contributions of the symbol sign. 
	Then we can rewrite the sign functions as follows:
	\begin{align}\label{signfunc}
		\sgn_{\ell}(C) = (-1)^{\sum_{a,b\in T} A_{\ell}(a, b)} \quad \text{and}\quad
		\mathrm{ssgn}(C) = (-1)^{\sum_{a,b\in T} B(a,b)}.
	\end{align}	
	
	Define the indicators ${\chi}_{i}(\ast): (a, b) \mapsto \{0,1\}$ 
	given by ${\chi}_{i}(\ast)(a,b) = [a_i \ast b_i]$, 
	where $\ast$ can be one of the binary relations $\{ <, =, >, \cdot \}$ and $\cdot$ denotes the complete relation (identical $1$). Similarly, we define $\xi(\ast): (a,b) \mapsto \{0,1\}$ given by $\xi(\ast)(a,b) = [C(a)\ast C(b)]$. 
	We will use that fact that our functions $F(\ast) \in \{0,1\}$ satisfy the identity
	\begin{align}\label{eqidentity}
		F(=) + F(<) + F(>) &= F(\cdot).
	\end{align}
	Denote $\chi_{i_{1}\ldots i_{\ell}}(\ast) := \chi_{i_{1}}(\ast)\ldots \chi_{i_{\ell}}(\ast)$. By exchanging the arguments $a$ and $b$ we can see that
	\begin{align}\label{revidentity}
		\chi_{i_{1}\ldots i_{\ell}}(<)\,
		\chi_{j_{1}\ldots j_{m}}(>)\,
		\chi_{k_{1}\ldots k_{s}}(=)\,
		\xi(<)
		=
		\chi_{i_{1}\ldots i_{\ell}}(>)\,
		\chi_{j_{1}\ldots j_{m}}(<)\,
		\chi_{k_{1}\ldots k_{s}}(=)\,
		\xi(>)
	\end{align}
	for any pairwise distinct indices $i_{1},\ldots,i_{\ell},j_{1},\ldots,j_{m},k_{1},\ldots,k_{s} \in [d]$. 
	  Then the above functions $A_{\ell}$ and $B$ can be rewritten as follows:
	\begin{align*}
	A_{\ell} &= \left(\sum_{i=1,i\neq \ell}^{d} \chi_{\ell}(=) \chi_{1}(=)\ldots\chi_{i-1}(=)\chi_{i}(<) \right) \xi(>),\\
		B &= \sum_{i=2}^{d} \chi_{1}(<)\chi_{i}(>)\xi(=) \\
		&= \sum_{i=2}^{d} \chi_{1}(<)\chi_{i}(>) (1 + \xi(<) + \xi(>))\\
		&= \sum_{i=2}^{d} \chi_{1}(<)\chi_{i}(>) + \left(\sum_{i=2}^{d} \chi_{1}(<)\chi_{i}(>) + \chi_{1}(>)\chi_{i}(<)\right)\xi(>).\\
	\end{align*}
	Denote $r = [d \text{ is odd}]$. Using the fact $\chi_{\ell}(=)\chi_{\ell}(<) = 0$ we obtain
	\begin{align*}
		\sum_{\ell=1+r}^{d} A_{\ell} 
		&= \sum_{\ell=1+r}^{d} \left(\sum_{i=1,i\neq \ell}^{d} \chi_{\ell}(=) \chi_{1\ldots i-1}(=)\chi_{i}(<) \right) \xi(>)\\
		&= \sum_{\ell=1+r}^{d} \left(\sum_{i=1}^{d} \chi_{\ell}(=) \chi_{1\ldots i-1}(=)\chi_{i}(<) \right) \xi(>)\\
		&= \left(
		\underbrace{(d-r)\sum_{i=1}^{d}  \chi_{1\ldots i-1}(=)\chi_{i}(<)}_{=0 \text{ as $(d-r)$ is even}} + \sum_{\ell=1+r}^{d}\sum_{i=1}^{d} \chi_{1\ldots i-1}(=)\chi_{i}(<) (\chi_{\ell}(<) + \chi_{\ell}(>))
			\right) \xi(>).
	\end{align*}
	Further collecting the terms with the multiple $\chi_{1}(=)$ separately we get
	\begin{align*}
		\sum_{\ell=1+r}^{d} A_{\ell} 
			&= \left(
			\sum_{\ell=2}^{d}\chi_{1}(<)(\chi_{\ell}(<)+\chi_{\ell}(>)) + (1-r)\chi_{1}(<)\right)\xi(>)\\
			&+\chi_{1}(=)\left(
		\underbrace{\sum_{i=2}^{d}\sum_{\ell=i+1}^{d} \chi_{2\ldots i-1}(=)\chi_{i}(<) (\chi_{\ell}(<) + \chi_{\ell}(>))}_{=: L} 
		+ \underbrace{\sum_{i=2}^{d}\chi_{2\ldots i-1}(=)\chi_{i}(<)}_{=: R}
			\right) \xi(>).
	\end{align*}
	Let us collect the terms in $R$ of the form $\chi_i(=)$ using the identity \eqref{eqidentity} consequently from right to left until the first non-dot relation appears, so that we can rewrite
	$$
		R = \sum_{i=2}^{d}\chi_{2\ldots i-1}(=)\chi_{i}(<) 
		= \sum_{i=2}^{d} \sum_{\ell=2}^{i-1}\chi_{2\ldots \ell-1}(=)(\chi_\ell(<) + \chi_\ell(>))\chi_{i}(<) + \sum_{i=2}^d \chi_i(<)
	$$
	and therefore by identity \eqref{revidentity} we have
	\begin{align*}
		\chi_1(=)(L + R - \sum_{i=2}^d \chi_i(<))\xi(>) &= \sum_{i=2}^{d} \sum_{\ell=i+1}^{d}\chi_{1\ldots i-1}(=)\left(
		\chi_i(<)\chi_\ell(>) + \chi_i(>)\chi_\ell(<)
		\right) \xi(>) \\
		&= \sum_{i=2}^{d} \sum_{\ell=i+1}^{d}\chi_{1\ldots i-1}(=)\chi_i(<)\chi_\ell(>) \left(
		\xi(<) + \xi(>)\right) \\
		&= \sum_{i=2}^{d} \sum_{\ell=i+1}^{d}\chi_{1\ldots i-1}(=)\chi_i(<)\chi_\ell(>) 
	\end{align*}
	where the latter equality is due to the fact $\chi_1(=)\xi(=) = 0$. Thus,
	\begin{align*}
		\sum_{\ell=1+r}^{d} A_{\ell} + B
			=& \left(
			\sum_{\ell=2}^{d}\chi_{1}(<)(\chi_{\ell}(<)+\chi_{\ell}(>)) + (1-r)\chi_{1}(<)\right)\xi(>) \\
			&+ \sum_{i=2}^d \chi_1(=)\chi_i(<)\xi(>)
			+ \sum_{i=2}^{d} \sum_{\ell=i+1}^{d}\chi_{1\ldots i-1}(=)\chi_i(<)\chi_\ell(>) \\
			&+ \sum_{i=2}^{d} \chi_{1}(<)\chi_{i}(>) + \left(\sum_{\ell=2}^{d} \chi_{1}(<)\chi_{\ell}(>) + \chi_{1}(>)\chi_{\ell}(<)\right)\xi(>) \\ 
			=& 
			\sum_{\ell=2}^{d}(\chi_{1}(<) + \chi_{1}(=) + \chi_{1}(>))\chi_{\ell}(<) \xi(>) + (1-r)\chi_{1}(<)\xi(>) \\
			&+ \sum_{i=2}^{d} \sum_{\ell=i+1}^{d}\chi_{1\ldots i-1}(=)\chi_i(<)\chi_\ell(>) + \sum_{i=2}^{d} \chi_{1}(<)\chi_{i}(>) \\ 
			=&
			\sum_{\ell=1+r}^{d}\chi_{\ell}(<) \xi(>) 
			+ \sum_{i=1}^{d} \sum_{\ell=i+1}^{d}\chi_{1\ldots i-1}(=)\chi_i(<)\chi_\ell(>).
	\end{align*}
	Finally, we compute the expressions while plugged into $\eqref{signfunc}$. The contribution of terms in the first sum in the latter expression can be computed explicitly 
	using the fact that each slice of $C$ contains each number in $[n]$ exactly once. i.e. $$\sum_{a,b\in T} (\chi_\ell(<)\xi(>))(a,b) = k^{2d-2}\binom{k}{2}\binom{n}{2}.$$
	There are $d-r$ equal contributions from this sum, and so they cancel out since $d-r$ is even. This shows that the signs product for $C$ does not depend on its values at cells, but only depends on its type. Then the second sum turns into
	\begin{align*}
		\sum_{a,b \in T}\left(
		\sum_{i=1}^{d} \sum_{\ell=i+1}^{d}\chi_{1\ldots i-1}(=) \chi_i(<)\chi_\ell(>)
		\right) (a,b)
		&= \sum_{a,b \in T} \left(
		\sum_{\ell=1}^{d} \chi_\ell(>) \sum_{i=1}^{d} \chi_{1\ldots i-1}(=)\chi_i(<)
		\right) (a,b) \\
		&= \sum_{a,b \in T} \sum_{\ell=1}^d [a < b] [a_\ell > b_\ell] \\
		&= \sum_{\ell=1}^d\mathrm{inv}(T_\ell).
	\end{align*}
	Therefore, 
	$$\sgn_1(C)^{d-1}\sgn_2(C)\ldots\sgn_d(C)\mathrm{ssgn}(C) = \mathrm{msgn}(T_1)\cdots\mathrm{msgn}(T_d) = \sgn(T)$$
	and the proof follows.
\end{proof}

\begin{corollary}\label{Latinsigns}
	Let $d > 2$. For every Latin hypercube $C \in \mathcal{C}^{(d)}_{k}$ 
	we have
	$$
		\sgn_{1}(C)^{d-1} \sgn_{2}(C) \cdots \sgn_{d}(C)\, \mathrm{ssgn}(C) = (-1)^{\lfloor \frac{d}{2} \rfloor \lfloor \frac{k}{2} \rfloor  k}.
	$$
\end{corollary}

\begin{proof}
	By Theorem~\ref{pLatinsigns} for $T=[k]^d$ we have 
	$$\sgn_{1}(C)^{d-1} \sgn_{2}(C) \cdots \sgn_{d}(C) = \sgn([k]^d).$$ 
	Note that $T_\ell = (1^{k^{d-\ell}}, \ldots, k^{k^{d-\ell}})^{k^{\ell-1}}$ for $\ell \in [d]$ (where we denote $x^n = (x,\ldots,x)$ is a sequence of $n$ elements and $(a_1,\ldots,a_n)^m = (a_1,\ldots,a_n,\ldots,a_1,\ldots,a_n)$ is a sequence of $nm$ elements). Then 
	$$
	\mathrm{inv}(T_\ell) = \binom{k^{\ell-1}}{2}\binom{k}{2} k^{2d-2\ell}.
	$$
	If $\ell$ is even, then $\binom{k^{\ell-1}}{2} \equiv \binom{k}{2} \equiv \lfloor \frac{k}{2} \rfloor \, (\textrm{mod}\ 2)$ and if $\ell$ is odd, then $\binom{k^{\ell-1}}{2} \equiv 0\, (\textrm{mod}\ 2)$. Hence,
	$$
		\sum_{\ell=1}^d \mathrm{inv}(T_\ell) \equiv \lfloor d/2 \rfloor \lfloor k/2 \rfloor k  \pmod{2} 
	$$
	and the statement follows.
\end{proof}

\begin{remark}
	Our proof is in the spirit of Janssen's proof \cite{jan95} showing that for every Latin square $L$ of length $k$ 
$$\sgn_{1}(L)\,\sgn_{2}(L)\,\mathrm{ssgn}(L) = (-1)^{\lfloor \frac{k}{2} \rfloor}.$$
\end{remark}

\section{Highest weight spaces}\label{sec:hwv0}
In this section we prove the following result. 
\begin{theorem}\label{thm81}
Let $d\ge 3$ be odd. If $AT_{d}(k) \ne 0$ for even $k$, then 
$g_d(n,k) > 0$ for all $n \le k^{d-1}$.
\end{theorem}
In particular, the condition $AT_3(k) \ne 0$ implies $g_3(n, k) > 0$ for all $n \le k^2$ and hence by Theorem~\ref{gthm}~(iii) also establishes the positivity $g_d(n,k) > 0$ for all $n \le k^{d-1}$. The proof is based on description of highest weight vectors. 

\subsection{Highest weight vectors}
The group $\Gamma = \mathrm{GL}(k)^{\times d}$ acts on the 
anti-symmetric space $\bigwedge^{m} (\mathbb{C}^{k})^{\otimes d}$ componentwise.
 A vector $v \in  \bigwedge^{m}(\mathbb{C}^{k})^{\otimes d}$ is called a {\it weight vector} if it is rescaled by the action of diagonal matrices, i.e. 
$$
	(\mathrm{diag}(a^{(1)}_{1},\ldots,a^{(1)}_{k}),\ldots, \mathrm{diag}(a^{(d)}_{1},\ldots,a^{(d)}_{k})) \cdot v = (a^{(1)})^{\lambda^{(1)}}\ldots (a^{(d)})^{\lambda^{(d)}} v,
$$ 
where $x^{\alpha} = x_{1}^{\alpha_{1}}\ldots x_{k}^{\alpha_{k}}$;
then $(\lambda^{(1)},\ldots,\lambda^{(d)})$ is the {\it weight} of $v$. The weight vectors of the same weight form a {\it weight space}. 
Any representation of $\Gamma$ can be decomposed into a direct sum of weight spaces. 
We now describe an explicit weight space decomposition from \cite{imw17}. 

We shall view the elements of $[k]^{d}$ ordered lexicographically. Let $P \subseteq [k]^{d}$.
Denote by $s_{\ell}(P, i)$ the number of elements in $P$ whose $\ell$-th coordinate is $i$, i.e. the number of points of $i$-th slice in $\ell$-th direction, and $s_{\ell}(P) := (s_{\ell}(P,1), \ldots, s_{\ell}(P, k))$ are the vectors of {\it marginals}. For $P = \{(x^{(1)}_{i},\ldots,x^{(d)}_{i}) : 1 \le i \le m \}$ we associate the vector
$$
	\psi_{P} := \bigwedge_{i=1}^{m} e_{x^{(1)}_{i}} \otimes\ldots\otimes e_{x^{(d)}_{i}}
$$
where $\{e_{i}\}_{i=1}^{k}$ is the standard basis of $\mathbb{C}^{k}$. The vectors $\{\psi_{P}\}$ over all possible $P \subseteq [k]^{d}$ of cardinality $m$ form a basis of $\bigwedge^{m}(\mathbb{C}^{k})^{\otimes d}$. Moreover, it is straightforward to verify that $\psi_{P}$ is a weight vector of weight $(s_1(P), \ldots, s_{d}(P))$ for the action of $\Gamma$.

Let $V(\lambda)$ be irreducible representation of $\mathrm{GL}(k)$ indexed by partition $\lambda$ (known as Weyl module). 

\begin{lemma}[cf. {\cite[Lemma~2.1]{imw17}}]
Let $\lambda^{(1)},\ldots,\lambda^{(d)}$ be partitions of size $m$ whose first parts are at most $k$. The Kronecker coefficient $g(\lambda^{(1)},\ldots,\lambda^{(d)})$ is equal to the multiplicity of the irreducible representation  $V((\lambda^{(1)})')\otimes\ldots\otimes V((\lambda^{(d)})')$ in the anti-symmetric space $\bigwedge^{m}(\mathbb{C}^k)^{\otimes d}$, where $\lambda'$ is the conjugate of $\lambda$.
\end{lemma}

Each irreducible representation $V((\lambda^{(1)})')\otimes\ldots\otimes V((\lambda^{(d)})')$ contains a unique one-dimensional space of weight $(\lambda^{(1)},\ldots,\lambda^{(d)})$ called the {\it highest weight space}, and so $g(\lambda^{(1)},\ldots,\lambda^{(d)})$ is the dimension of the corresponding highest weight space. The {\it highest weight vectors}, which are elements of the highest weight space, can be characterized by the action of the Lie algebra $\mathfrak{g}$ 
of $\Gamma$. 
Let $A \in \mathfrak{g}$ and $I$ be the unit of $\Gamma$. 
Then $\epsilon A+I \in \Gamma$ for sufficiently small $\epsilon$ and 
the action of $A \in \mathfrak{g}$ on $v$ is defined by $A v := \lim_{\epsilon \to 0}((\epsilon A+I)v - v)\epsilon^{-1}$. 
Let $E_{i,j}$ be the $k \times k$ matrix with a single $1$ at entry $(i,j)$ and zero elsewhere. For $i < j$ the operator $v \mapsto (E_{i,j},0,0,\ldots) v$ is called a {\it raising operator} in direction $1$ (which is defined similarly for other directions). Then a weight vector vanishing by all raising operators is a highest weight vector. 

We are interested in the case 
$\lambda^{(1)}=\ldots=\lambda^{(d)} = n \times k$ motivated by the questions in previous sections. Hence we need to consider the weight spaces of weight $(k \times n)^{d}$. That is, the span of vectors $\psi_{P}$ where $P$ runs over magic sets (defined in the previous section), such that $P$ has $n$ elements in each of $k$ slices in any direction.
Define the following subspaces of $\bigwedge^{nk}(\mathbb{C}^{k})^{\otimes d}$:
$$
	\mathcal{B}_{d,k}(n) := \mathrm{span}\{\psi_{P} \mid P \subseteq [k]^d \text{ is a magic set of cardinality }kn \} 
$$
for $n = 0,\ldots,k^{d-1}$. Each $\mathcal{B}_{d,k}(n)$ is the weight space of weight $(k \times n)^{d}$. In particular, $\mathcal{B}_{d,k}(1)$ is a vector space with the basis indexed by permutation hypermatrices and $\mathcal{B}_{d,k}(k^{d-1})$ is a one-dimensional vector space with a single basis vector $\psi_{[k]^{d}}$. 

The action of raising operators on the highest weight subspace of the space $\mathcal{B}_{d,k}(n)$ can be described combinatorially. The generators (in the direction $1$) are of the form $E^{(1)}_{i,i+1} := (E_{i,i+1},0,0,\ldots)$. 
The action is linear and defined on basis vectors as follows:
	$$
		E^{(1)}_{i,i+1} \psi_{P} = \sum_{j=1}^{n} 
		(-1)^{j+1} \delta_{x^{(1)}_{j}, i+1} \,
		e_{i}\otimes e_{x^{(2)}_{j}} \otimes \ldots \otimes e_{x^{(d)}_{j}} 
		\wedge 
		\bigwedge_{\ell=1, \ell \neq j}^{n}  e_{x^{(1)}_{\ell}}\otimes \ldots \otimes e_{x^{(d)}_{\ell}}
	$$
	for any $P = \{ (x^{(1)}_{1},\ldots,x^{(d)}_{1}),\ldots,(x^{(1)}_{n},\ldots,x^{(d)}_{n}) \}$ and $i \in [k-1]$. In other words, for each $p \in P$ with the first coordinate $i+1$, we replace $p$ with the point $p' = p - (1,0,0,\ldots)$ without changing the order of elements in $P$. Denote the new (ordered) set by $P-p+p'$. Then the action of the raising operator can also be written as follows: 
	$$E^{(1)}_{i,i+1} \psi_{P} = \sum_{p = (i+1,\ldots) \in P} \psi_{P-p+p'}.$$
	Note that $P-p+p'$ has $n$ points in each slice except the slices $i$ and $i+1$ in the direction $1$ where it has $n+1$ and $n-1$ elements, respectively. 
	The action or raising operators  in other directions is defined similarly. 
	
	We can combine all raising operators into a single operator $E$ given by $E := \sum_{j=1}^{d}\sum_{i=1}^{k-1} E^{(j)}_{i,i+1}$. Since the images of $E^{(j)}_{i,i+1}$ do not intersect, we have:
	$$v \in \mathcal{B}_{d,k}(n) \text{ is highest a weight vector} \iff E(v) = 0.$$
Alternatively, $v \in \ker E$. Thus, $\dim \ker E|_{\mathcal{B}_{d,k}(n)} = g_d(n,k)$.

\begin{lemma}\label{lemma:ab}
Let $\alpha \in \mathcal{B}_{d,k}(n)$ and $\beta \in \mathcal{B}_{d,k}(m)$ be highest weight vectors and assume $\alpha \wedge \beta \ne 0$. Then $\alpha \wedge \beta \in \mathcal{B}_{d,k}(n + m)$ is also a highest weight vector.
\end{lemma}

\begin{proof}
Checking the action of raising operators $E$ we have 
$$
E(\alpha \wedge \beta) = E(\alpha) \wedge \beta + \alpha \wedge E(\beta) = 0
$$
and the statement follows.
\end{proof}

\subsection{A hyperdeterminantal form} 
Define the following special {\it hyperdeterminantal} element of the space $\mathcal{B}_{d,k}(1)$:
\begin{align}\label{Ddef}
	\omega := \sum_{\pi_{2},\ldots,\pi_{d} \in S_{k}} \sgn(\pi_{2}\ldots \pi_{d}) \bigwedge_{i=1}^{k} e_{i}\otimes e_{\pi_{2}(i)} \otimes\ldots\otimes e_{\pi_{d}(i)}  \in \mathcal{B}_{d,k}(1).
\end{align}

We shall write 
$$
\omega^n := \underbrace{\omega \wedge \cdots \wedge \omega}_{n \text{ times}} \in \mathcal{B}_{d,k}(n).
$$

\begin{theorem}\label{thm:omega}
Let $d \ge 3$ be odd. We have:

(i)  $\omega$ is a unique (up to a scalar) highest weight vector in $\mathcal{B}_{d,k}(1)$. 

(ii) For $k$ even and 
$n \le k^{d-1}$, 
$$
\omega^{n} = \pm \sum_{T} \sgn(T) AT_d(k,T)\, \psi_{T},
$$
where the sum is over all magic sets $T \subseteq [k]^d$ of cardinality $|T| = nk$. 
In particular, for $T = [k]^d$ 
we obtain
$$
\omega^{k^{d - 1}} = \pm AT_d(k)\, \psi_{[k]^d},
$$

(iii) For $k$ odd, $\omega^2 = 0.$
\end{theorem}

\begin{proof}
Let $\mathcal{S} = \{ \pi=(\pi_{1},\ldots,\pi_{d}): \pi \in (S_{k})^d, ~\pi_{1}=\mathrm{id} \}$. 
	For $\pi = (\pi_1,\ldots, \pi_d) \in \mathcal{S}$ denote 
	$\sgn(\pi) := \sgn(\pi_{1}\cdots\pi_{d})$ and 
	$\pi(i) = (\pi_{1}(i),\ldots,\pi_{d}(i)) \in [k]^{d}$. Let us also denote 
	$e_{\pi(i)} := e_{\pi_{1}(i)}\otimes \cdots \otimes e_{\pi_{d}(i)}$ 
	and $e_{\pi}:= \bigwedge_{i=1}^{k}e_{\pi(i)}$. 
	
(i) 	Firstly note that $g_{d}(1, k) = 1$ and hence there is only one highest weight vector in $\mathcal{B}_{d,k}(1)$. Now we need to verify that the action of raising operators vanishes $\omega$. Let us check it on the generators of the form $E^{(1)}_{i,i+1} = (E_{i,i+1},0,0,\ldots)$. 

		For $\pi \in \mathcal{S}$ let $L_{\pi} = \{\pi(1),\ldots,\pi(i-1)\}$, $a_{\pi} = \pi(i)$, $b_{\pi} = \pi(i+1)$ and $R_{\pi} = \{\pi(i+2),\ldots,\pi(k)\}$.
	Applying the raising operator to $\omega$ we get
	\begin{align}\label{actD}
		(E_{i,i+1},0,\ldots,0)\, \omega = \sum_{\pi \in \mathcal{S}}\sgn(\pi)\, \psi_{\{L_{\pi},a_{\pi},b'_{\pi}, R_{\pi}\}}
	\end{align}
	where the set considered as ordered and for $x\in[k]^{d}$ we denote $x' = x - (1,0,\ldots,0)$. Note that the cells $a_{\pi}$ and $b'_{\pi}$ both lie in the  $i$-th slice of the first direction in $[k]^{d}$, and so the last sum contains only wedge products $\psi_{S}$ indexed by (ordered) sets $S$ having two elements in the slice $i$ and no elements in the slice $i+1$. Let us fix a set $P = L \cup \{a', b'\} \cup R$ where $L$ and $R$ are (lexicographically ordered) sets of cells in the slices $1,\ldots,i-1$ and $i+2,\ldots,k$ respectively, and $a < b$ are cells from the slice $i+1$. Consider the coefficient $c_{P}$ at the vector $\psi_{P}$ in the sum \eqref{actD}. There are exactly two tuples of permutations $\pi^{1}$ and $\pi^{2}$ in $\mathcal{S}$ which contribute to coefficient at $\psi_P$, defined as $L_{\pi^{1}} = L_{\pi^{2}} = L$, $R_{\pi^{1}} = R_{\pi^{2}} = R$ and 
	\begin{align*}
	&a_{\pi^{1}}=\pi^{1}(i)=a', \quad b_{\pi^{1}}=\pi^{1}(i+1)=b, \\
	&a_{\pi^{2}}=\pi^{2}(i) = b'~,\quad b_{\pi^{2}}=\pi^{2}(i+1) = a.
	\end{align*}	
	First, we observe that $c_{P} = \sgn(\pi^{1}) - \sgn(\pi^{2})$, since 
	$$
		c_{P}\, \psi_{P} = \sgn(\pi^{1})\, \psi_{\{L, a', b', R\}} + \sgn(\pi^{2})\, \psi_{\{L, b', a', R\}} = (\sgn(\pi^{1}) - \sgn(\pi^{2}))\, \psi_{P}.
	$$	
	Second, let us check that $\sgn(\pi^{1}) = \sgn(\pi^{2})$. Indeed, 
	$$\pi^{1} = (\mathrm{id}, \pi^{1}_{2},\ldots,\pi^{1}_{d}) = (\mathrm{id},(i,i+1)\cdot \pi^{2}_{2}, \ldots, (i,i+1)\cdot\pi^{2}_{d}),$$ 
	where $(i,i+1) \in S_{k}$ is a transposition, 
	and hence, $\sgn(\pi^{1}) = (-1)^{d-1}\sgn(\pi^{2}) = \sgn(\pi^{2})$. Therefore, $c_{P} = 0$ and since $P$ was chosen arbitrarily, we obtain that $(E_{i,i+1},0,\ldots,0)\, \omega = 0$. The same argument works in other directions by rewriting $\omega$ as follows. 
	For odd $d$ the choice of the fixed first direction $\pi_{1} = id$ can be changed, namely, for any $\ell \in [d]$ we have 
\begin{align*}
		\omega 
		&= \sum_{\pi_{2},\ldots,\pi_{d}\in S_{k}} \sgn(\pi_{2}\ldots\pi_{d})\bigwedge_{i=1}^{k}e_{i}\otimes e_{\pi_{2}(i)} \otimes\ldots\otimes e_{\pi_{d}(i)} \\
		&= \sum_{\pi_{2},\ldots,\pi_{d}\in S_{k}} \sgn(\pi_{\ell})^{d-1}\sgn(\pi_{2}\ldots\widehat{\pi_{\ell}}\ldots\pi_{d})\bigwedge_{i=1}^{k}e_{\pi_{\ell}^{-1}(i)}\otimes \ldots\otimes e_{i}\otimes \ldots \otimes e_{\pi^{-1}_{\ell}\pi_{d}(i)} \\
		&= \sum_{\pi_{1},\ldots,\pi_{\ell-1},\pi_{\ell+1},\ldots,\pi_{d}\in S_{k}} \sgn(\pi_{1}\ldots\widehat{\pi_{\ell}}\ldots\pi_{d})\bigwedge_{i=1}^{k}e_{\pi_{1}(i)}\otimes\ldots\otimes e_{i} \otimes\ldots\otimes e_{\pi_{d}(i)}
\end{align*}
(where as usual $~\widehat{\cdot}~$ denotes the absence in the sequence). 

\

(ii) For each $\pi \in \mathcal{S}$ the set ${d}({\pi}) := \{ \pi(1),\ldots,\pi(k) \}$ is a diagonal of  $[k]^{d}$. 
	So we have 
	$
		\omega = \sum_{\pi \in \mathcal{S}} \sgn(\pi)\, e_{\pi}
	$
	where each term with index $\pi$ in the sum corresponds to the diagonal $d({\pi})$. 
	Let us denote $$M_{d,k}(n) :=  \left\{T \subseteq [k]^d : T \text{ is a magic set of cardinality } nk  \right\}.$$
	Consider the expansion 
	\begin{align*}
		\omega^{n} &= \left( \sum_{\pi \in \mathcal{S}} \sgn(\pi)\bigwedge_{i=1}^{k} e_{\pi(i)} \right)^{\wedge n} \\
		&= \sum_{\pi^{1},\ldots,\pi^{n} \in \mathcal{S}}\sgn(\pi^1)\cdots\sgn(\pi^n)\, e_{\pi^{1}}\wedge\ldots\wedge e_{\pi^{n}} \\
		&= 
		\sum_{T\in M_{d,k}(n)}~
		\sum_{\pi^{1},\ldots,\pi^{n} \in \mathcal{S}:~
			\cup_{i=1}^k d({\pi^i}) = T} 
		\sgn(\pi^1)\cdots\sgn(\pi^n) \,
		e_{\pi^{1}}\wedge\ldots\wedge e_{\pi^{n}} 
	\end{align*}	
	since if for some $i<j$ we have $d({\pi^{i}}) \cap d({\pi^{j}}) \neq \varnothing$ then a corresponding term vanishes in the above expansion due to the property of the wedge product; hence we may assume that  $d({\pi^{1}}),\ldots,d({\pi^{n}})$ form a partition of $T$ into diagonals. We can rewrite this sum as follows:
	\begin{align*}
		\omega^n = \sum_{T\in M_{d,k}(n)} ~\sum_{\pi^{1},\ldots,\pi^{n} \in \mathcal{S}:~ \cup_{i=1}^n d({\pi^i}) = T} \sgn(\pi^{1})\cdots\sgn(\pi^{n})\, 
		&e_{\pi^{1}(1)}\wedge\cdots\wedge e_{\pi^{1}(k)}\wedge\\
		&e_{\pi^{2}(1)}\wedge\cdots\wedge e_{\pi^{2}(k)}\wedge\\
		&\ldots\\
		&e_{\pi^{n}(1)}\wedge\cdots\wedge e_{\pi^{n}(k)}.
	\end{align*}  
	Let us fix $T\in M_{d,k}(n)$ and $\pi^{1},\ldots,\pi^{n} \in \mathcal{S}$ such that $\bigcup_{i=1}^{n} d(\pi^{i})=T$ is a set partition. We then construct a partial Latin hypercube $C$ as follows: for each $i\in [n]$ set $C(d({\pi^{i}})) = i$, i.e. each cell of the diagonal $d({\pi^{i}})$ has value $i$ in $C$, and $0$ otherwise. 
	Then by definition $$\mathrm{ssgn}(C) = \sgn(\pi^{1})\cdots\sgn(\pi^{n}).$$ Let $\sigma_{j}$ be the permutation formed by values of $C$ in $j$-th slice of the first direction collected in lexicographical order, then $$\sgn_{1}(C) = \sgn(\sigma_{1})\cdots\sgn(\sigma_{k}).$$ Since the union of the diagonals covers $T$, the corresponding vector term is up to a sign equal to $\psi_{T}$. Let us describe the canonical order of wedge multiples. Note that for a fixed $j\in[k]$ the cells $\pi^{1}(j),\ldots,\pi^{n}(j) \in [k]^{d}$ form the full $j$-th slice of $T$ in the first direction. Let us reorder the wedge multiples $e_{\pi^{1}(j)}, \ldots, e_{\pi^{n}(j)}$ lexicographically by indices, this is achieved via the permutation $\sigma^{-1}_{j}$. Let $a_{j,1},\ldots,a_{j,n} \in T$ be the elements of $T$ in $j$-th slice, written in lexicographical order. Then we have
	$$
	e_{\pi^{1}(j)} \wedge \ldots \wedge e_{\pi^{n}(j)} = \sgn(\sigma_{j})\, e_{a_{j,1}}\wedge e_{a_{j,2}} \wedge \ldots \wedge  e_{a_{j,n}}.
	$$
	Denote the vector term
	\begin{align*}
		\phi_T := ~ &e_{a_{1,1}}\wedge\ldots\wedge e_{a_{k,1}}\wedge\\
		&e_{a_{2,n}}\wedge\ldots\wedge e_{a_{k,2}}\wedge\\
		&\ldots\\
		&e_{a_{1,n}}\wedge\ldots\wedge e_{a_{k,n}}.
	\end{align*}
	Then we obtain
	\begin{align*}
		\omega^{n} = \sum_{T\in M_{d,k}(n)}\left(\sum_{C \in \mathcal{C}_{d}(k,T)} \mathrm{ssgn}(C)\, \sgn_1(C) 
		\right) \cdot  ~ \phi_T.
	\end{align*}
	By Theorem~\ref{pLatinsigns} for odd $d$ 
	we have 
	$$\mathrm{sgn}_{1}(C)\, \mathrm{ssgn}(C) = \mathrm{sgn}_{1}(C)\,\mathrm{sgn}_{2}(C) \cdots \mathrm{sgn}_{d}(C)\, \sgn(T)= \sgn(C)\, \sgn(T).$$ 
	Note that $\phi_{T} = (-1)^{\ell} \psi_{T}$ fo any $T$, where $\ell$ is a constant which does not depend on $T$. 
	Therefore, 
	$$
		\omega^{n} = \sum_{T\in M_{d,k}(n)} \sgn(T)\,  AT_{d}(k,T)\, \phi_T = \pm \sum_{T\in M_{d,k}(n)} \sgn(T)\, AT_{d}(k,T)\, \psi_T.
	$$

\

	(iii) Let us now show that if $k$ is odd, then $\omega^2 = 0$. 
		Choose an arbitrary total order $<$ on $\mathcal{S}$. A straightforward calculation shows that
	\begin{align*}
		\omega^{2} = \sum_{\pi^{1}, \pi^{2} \in \mathcal{S}, \pi^{1} < \pi^{2}} \sgn(\pi^{1})\, \sgn(\pi^{2})\, (e_{\pi^{1}}\wedge e_{\pi^{2}} + e_{\pi^{2}}\wedge e_{\pi^{1}}) = 0
	\end{align*}
	where we used the fact that $e_{\pi^{2}}\wedge e_{\pi^{1}} = (-1)^{k}e_{\pi^{1}}\wedge e_{\pi^{2}} = -e_{\pi^{1}}\wedge e_{\pi^{2}}$. 
\end{proof}

\begin{corollary}\label{cor85}
Let $d \ge 3$ be odd and $2 \le n \le k^{d-1}$. The following statements are equivalent: 

(1) $\omega^n \ne 0$

(2) $AT_d(k, T) = \Delta_T(I_n) \ne 0$  for some magic set $T \subseteq [k]^d$ of cardinality $nk$ (or a $d \times nk$ balanced table)

(3) $\omega^i \in \mathcal{B}_{d,k}(i)$ is a highest weight vector for all $1 \le i \le n$.
\end{corollary}

\begin{proof}
(1) $\iff$ (2) follows directly from the expansion in (ii).  

The condition $\omega^n \ne 0$ is equivalent to $\omega^i \ne 0$ for all $i \le n$. Since $\omega$ is a highest weight vector, by Lemma~\ref{lemma:ab}, if $\omega^i \ne 0$ then it is a highest weight vector for all $i \le n$. 
\end{proof}

\begin{corollary}\label{cor:atob}
Let $d\ge 3$ be odd and assume $AT_{d}(k) \ne 0$ for even $k$. Then $\omega^{n} \in \mathcal{B}_{d,k}(n)$ is a highest weight vector for all $1 \le n \le k^{d-1}$. In particular, $g_d(n,k) > 0$ for all $n \le k^{d-1}$.
\end{corollary}

This establishes Theorem~\ref{thm81}.


We also have the following weaker statement conditional on the Alon--Tarsi conjecture $AT(k) = AT_2(k) \ne 0$ for even $k$.

\begin{corollary}\label{cor:atk}
Let $d \ge 3$ be odd. Assume $AT(k) \ne 0$ for even $k$. Then $g_{d}(n,k) > 0$ for all $n \le k$. 
\end{corollary}
\begin{proof}
Let us show that $\omega^k \ne 0$ and hence $\omega^n$ is a highest weight vector for all $n \le k$. By Proposition~\ref{ATtable} for some $d \times k^2$ balanced table $\Delta_T(I_k) \ne 0$ is equivalent to $AT(k) \ne 0$. Hence, $\omega^k \ne 0$ by the results above. 
\end{proof}

\begin{remark}\label{rem:atk}
Hence, $AT(k) \ne 0$ also implies that $g_d(n,k) = g_{d}(k^{d-1} - n, k) > 0$ for $n \le k$ and as noted in Remark~\ref{rmkd} we get $\delta_{d}(k^{d - 1} - n) = (k^{d-1} -n) k$ for all $n \le k$. Of course,  $AT_3(k) \ne 0$ gives much more information on the degree sequence $\delta_d(n)$ as discussed in the introduction.
\end{remark}

\begin{remark}
For $d = 3$ analogous statements on positivity of Kronecker coefficients as in the last corollary can also be obtained from results in \cite{kumar}, see \cite{burg}.
\end{remark}

\begin{remark}
For $n > 1$ let $k_d(n)$ be minimal number $k$ for which there is an invariant $\Delta_T$ of degree $kn$ such that $\Delta_T(I_n) \ne 0$. Note that $k_d(n) \ge \delta_d(n)/n$ and $k_d(n)$ is even, cf. Remark~\ref{remstab}. By Corollary~\ref{cor85} if $\omega^{n} \ne 0$ then $\omega^{n - 1} \ne 0$, and hence if there is an invariant $\Delta_T$ of degree $nk$ with $\Delta_T(I_n) \ne 0$, then there is also an invariant $\Delta_{T'}$ of degree $(n-1)k$ with $\Delta_{T'}(I_{n-1}) \ne 0$. This implies the monotonicity $k_d(n) \ge k_d(n-1)$. Note that such property fails for $\delta_3(n)/n$, e.g. $\delta_3(8)/8 = 3 < \delta_3(7)/7 = 4$.
\end{remark}

\begin{remark}\label{remod}
The dual of the element $\omega$ is actually Cayley's first hyperdeterminant, cf.~Example~\ref{cayley1h}. Let $d$ be odd. For any $x_{1},\ldots,x_{k} \in (\mathbb{C}^{k})^{\otimes d}$ let $X\in (\mathbb{C}^{k})^{\otimes (d+1)}$ be tensor formed via concatenation of $x_{i}$ by slices, i.e. $X = \sum_{i=1}^{n} e_{i} \otimes x_{i}$. Then 
$$
\omega^*(x_1,\ldots, x_k) = \frac{1}{k!} \sum_{\sigma_1,\ldots, \sigma_d \in S_k} \sgn(\sigma_1 \cdots \sigma_d) \prod_{i = 1}^k X_{i, \sigma_1(i), \ldots, \sigma_d(i)}.
$$
The element $(\omega^{*})^{\wedge n}$ can be regarded as a multi-linear skew-symmetric $n k$-form on the space $(\mathbb{C}^{k})^{\otimes d}$ for $n \le k^{d-1}$, as well as a polynomial function on $(\mathbb{C}^{k})^{d+1}$ which is multi-linear and skew-symmetric in slices. This $n k$-form can be calculated as follows:
\begin{align*}
	(\omega^n)^{*}(x_{1},\ldots,x_{n k}) = \sum_{I_{1},\ldots,I_{n}}  \varepsilon_{I_{1}\ldots I_{n}}\, \omega^{*}(x_{i_{11}},\ldots,x_{i_{1k}}) \cdots \omega^{*}(x_{i_{n1}},\ldots,x_{i_{nk}}) 
\end{align*}
where the sum is over subsets $I_{1},\ldots,I_{n} \subseteq [n k]$ with $I_{j} = \{i_{j1} < \ldots < i_{jk}\}$ and $I_1\cup \ldots \cup I_{n} = [n k]$, 
and $\varepsilon_{I_{1},\ldots,I_{n}} = \sgn(i_{11},\ldots,i_{1k}, \ldots, i_{n 1}, \ldots, i_{n k})$.
\end{remark}

\subsection{Unconditional Kronecker positivity} Finally, we show that using the conditional statement in Corollary~\ref{cor:atk} we can obtain the following unconditional result.

\begin{theorem}\label{uncon}
Let $d \ge 3$ be odd and $k$ be even. Then $g_{d}(n, k) > 0$ for all $n \le \sqrt{k}/2 - 1$.
\end{theorem}

\begin{proof}
It is known that for every prime $p \ge 3$ we have $AT(p \pm 1) \ne 0$ (see \cite{dri1, glynn}) and therefore from Corollary~\ref{cor:atk} we have $g_d(n, p\pm 1) > 0$ for all $n \le p \pm 1$. Hence, using the semigroup property of Kronecker coefficients 
we get that $g_d(n, a(p-1) + b(p+1)) > 0$ for any $a, b \in \mathbb{N}$ and all $n \le p - 1$.

For $k \le 8$ the statement can be checked computationally. For $k \ge 10$, we can choose a prime $p \in [\sqrt{k}/2, \sqrt{k}]$ and write $k = \ell p + r$  so that $r \in [0, p-1]$ and $\ell - r$ is even; note that $\ell - r > 0$ since $\ell + 1> k/p  \ge p > r$. Hence we can present $k$ as follows:
$$
k = \ell p + r 
= \frac{\ell + r}{2}(p+1) + \frac{\ell - r}{2}(p-1).
$$
Therefore, using the semigroup property as noted above we obtain that 
$
g_{d}(n, k) > 0
$
for all $n \le p - 1$ and the result follows.
\end{proof}

\begin{remark}
Since it is known that for any sufficiently large $k$ there is a prime in the interval $[\sqrt{k}- o(\sqrt{k}), \sqrt{k}]$, we can get a better bound that $g_{d}(n, k) > 0$ for all $n \le \sqrt{k} - o(\sqrt{k})$.  
\end{remark}


\begin{corollary}
Let $d \ge 3$ be odd and $k$ be even. We have $\delta_d(n) = nk$ for all $k^{d - 1} - \sqrt{k}/2 + 1 \le n \le k^{d -1}$.
\end{corollary}

This completes part of Theorem~\ref{thm:one}.



\newpage

\appendix


\section{Some tables}\label{appa}

\begin{table}[h]
\caption{The table of $\delta_3(n)/n$ for $ 1 \le n \le 16$.} 
{\small
\begin{center}
\begin{tabular}{ c | c c c c c c c c c c c c c c c c} 
  $n$ & 1 & 2 & 3 & 4 & 5 & 6 & 7 & 8 & 9 & 10 & 11 & 12 & 13 & 14 & 15 & 16 \\ 
  \hline
  $\delta_3(n)/n$ & 1 & 2 & 2 & 2 & 3 & 3 & 4 & 3 & 3 & 4 & 4 & 4 & 4 & 4 & 4 & 4
\end{tabular}
\end{center}
}
\end{table}



\begin{table}[h]
\caption{The table of {$g_{d}(n, 2) = \dim \mathrm{Inv}_d(n)_{2n}$} for $0 \le n \le 16$, $d = 3,5,7$. Note that $g_{d}(n,2) = g_{d}(2^{d-1} - n, 2)$ and $g_d(n,2) = 0$ for $n > 2^{d-1}$.}
{\small
\begin{tabular}{ c | l l l} 
  $n \backslash d$ & 3 & 5 & 7 \\ 
  \hline
  $0$ & 1 & 1 & 1 \\
  $1$ & 1 & 1 & 1 \\
  $2$ & 1 & 5 & 21\\
  $3$ & 1 & 11 & 161\\
  $4$ & 1 & 35 & 3341\\
  $5$ & 0 & 52 & 64799\\
  $6$ & 0 & 112 & 1407329\\
  $7$ & 0 & 130 & 27536390\\
  $8$ & 0 & 166 & 482181504\\
  $ 9$ & 0 & 130 & 7403718609\\
  $10$ & 0 & 112 & 99468725538\\
  $11$ & 0 & 52 & 1168191022248\\
  $12$ & 0 & 35 & 12009002387858\\
  $13$ & 0 & 11 & 108266717444858\\
  $14$ & 0 & 5 & 857991447205123\\
  $15$ & 0 & 1 & 5991301282600760\\
  $16$ & 0 & 1 & 36953889463653995\\
\end{tabular}
}
\end{table}

\begin{table}[h]
\caption{The table of {$g_{d}(n, 3) = \dim \mathrm{Inv}_d(n)_{3n}$} for $0 \le n \le 9$, $d = 3,5,7$. Note that $g_{d}(n,3) = g_{d}(3^{d - 1} - n, n)$ and $g_{d}(n,3) = 0$ for $n > 3^{d - 1}$.}
{\small
\begin{tabular}{  c | l l l } 
  $n \backslash d$ & 3 & 5 & 7 \\ 
  \hline
  $0$ & 1 & 1 & 1 \\
  $1$ & 1 & 1 & 1 \\
  $2$ & 0 & 1 & 70\\
  $3$ & 1 & 385 & 636177\\
  $4$ & 1 & 44430 & 9379255543\\
  $5$ & 1 & 5942330 & 215546990657498\\
  $6$ & 1 & 781763535 & 6136455833113627910\\
  $7$ & 0 & 93642949102 & 191473697724924688999920\\
  $8$ & 1 & 9856162505065 & 6100591257296003780834337810\\
  $9$ & 1 & 894587378523908 & 190121112332748795911599731191284\\
\end{tabular}


}
\end{table}

\begin{table}[h]
\caption{The table of {$g_{3}(n, k) = \dim \mathrm{Inv}_3(n)_{kn} = g(n\times k, n \times k, n\times k)$} for $1 \le k \le 6$, $1 \le n \le 8$. Note that $g_{3}(n, k) = g_{3}(k^2 - n, k)$ and $g_{3}(n, k) = 0$ for $n > k^2$.}
{\small
\begin{tabular}{  c | l l l l l l l l} 
  $k \backslash n$ & 1 & 2 & 3 & 4 & 5 & 6 & 7 & 8  \\ \hline
  $1$ & 1 & 0 & 0 & 0 & 0 & 0 & 0 & 0  \\ 
  $2$ & 1 & 1 & 1 & 1 & 0 & 0 & 0 & 0  \\ 
  $3$ & 1 & 0 & 1 & 1 & 1 & 1 & 0 & 1  \\ 
  $4$ & 1 & 1 & 2 & 5 & 6 & 13 & 14 & 18  \\ 
  $5$ & 1 & 0 & 1 & 4 & 21 & 158 & 1456 & 9854  \\ 
  $6$ & 1 & 1 & 3 & 16 & 216 & 9309 & 438744 & 17957625  
\end{tabular}
}
\end{table}

\

\


\begin{table}[h]
\caption{The table of $g_5(n,k) = \dim \mathrm{Inv}_5(n)_{kn}$ for $1 \le k \le 5$, $1 \le n \le 6$. }
{\small
\begin{tabular}{ c | l l l l l l} 
  $k \backslash n$ & 1 & 2 & 3 & 4 & 5 & 6 \\ \hline
  $1$ & 1 & 0 & 0 & 0 & 0 & 0 \\ 
  $2$ & 1 & 5 & 11 & 35 & 52 & 112 \\ 
  $3$ & 1 & 1 & 385 & 44430 & 5942330 & 781763535 \\ 
  $4$ & 1 & 36 & 44522 & 381857353 & 5219755745322 & 87252488565829772   \\ 
  $5$ & 1 & 15 & 6008140 & 5220537438711 & 10916817688177999825 & 36929519748583464067841925   \\ 
\end{tabular}
}
\end{table}

\

\begin{table}[h]
\caption{The table of $g_7(n,k) = \dim \mathrm{Inv}_7(n)_{kn}$ for $1 \le k \le 5$, $1 \le n \le 5$.}
{\small
\begin{tabular}{ c | l l l l l} 
  $k \backslash n$ & 1 & 2 & 3 & 4 & 5  \\ \hline
  $1$ & 1 & 0 & 0 & 0 & 0  \\ 
  $2$ & 1 & 21 & 161 & 3341 & 64799  \\ 
  $3$ & 1 & 70 & 636177 & 9379255543 & 215546990657498 \\ 
  $4$ & 1 & 3362 & 9379321798 & 220746106806871065 & 14446465578705208466014240 \\ 
  $5$ & 1 & 62204 & 215601786541974 & 14446471715159302533654142 & {5370640146091973101847897273759375505}  \\ 
\end{tabular}
}
\end{table}

\

\begin{table}[h]
\caption{The table of $g_9(n,k) = \dim \mathrm{Inv}_9(n)_{kn}$ for $1 \le k \le 5$, $1 \le n \le 4$.}
{\small 
\begin{tabular}{ c | l l l l } 
  $k \backslash n$ & 1 & 2 & 3 & 4  \\ \hline 
  $1$ & 1 & 0 & 0 & 0  \\ 
  $2$ & 1 & 85 & 3151 & 538525   \\ 
  $3$ & 1 & 2331 & 1120708625 & 2001892792552290  \\ 
  $4$ & 1 & 538882 & 2001892904676744 & {127405523125364290107394265} \\ 
  $5$ & 1 & 112063889 & 7776236632415272731072 & {39943390428604984740410365287203655408}   \\ 
\end{tabular}
}
\end{table}


\newpage

\begin{thebibliography}{abcdefghi}
\bibitem[AT92]{at}
N. Alon and M. Tarsi, Coloring and orientations of graphs, Combinatorica {\bf 12} (1992), 125--143.

\bibitem[AY21]{ay}
A. Amanov and D. Yeliussizov, Tensor slice rank and Cayley's first hyperdeterminant,  arXiv:2107.08864 (2021).

\bibitem[BB04]{bb04}
C. Bessenrodt and C. Behns, On the Durfee size of Kronecker products of characters of the symmetric group and its double covers, J. Algebra {\bf 280} (2004), 132--144.

\bibitem[Bre12]{br12}
M. Bremner, On the hyperdeterminant for $2 \times 2 \times 3$ arrays, Linear Multilinear Algebra {\bf 60} (2012), 921--932.

\bibitem[BBS12]{bbs12}
M. Bremner, M. Bickis, M. Soltanifar, Cayley’s hyperdeterminant: A combinatorial approach via representation theory, Linear Algebra Appl. {\bf 437} (2012), 94--112.

\bibitem[BH13]{brh13}
M. Bremner, J. Hu, Fundamental invariants for the action of ${{SL_3 (\mathbb {C})\times SL_3 (\mathbb {C})\times SL_3 (\mathbb {C})}} $ on $3 \times 3 \times 3$ arrays, Math. Comp. {\bf 82} (2013), 2421--2438.

\bibitem[BHO14]{bro14}
M. Bremner  J. Hu, L. Oeding, The $3 \times 3 \times 3$ hyperdeterminant as a polynomial in the fundamental Invariants for ${{SL_3 (\mathbb {C})\times SL_3 (\mathbb {C})\times SL_3 (\mathbb {C})}}$, Math. Comput. Sci. {\bf 8} (2014), 147--156.

\bibitem[BVZ10]{bvz10}
A. Brown, S. Van Willigenburg, M. Zabrocki, Expressions for Catalan Kronecker products. Pacific J. Math. {\bf 248} (2010), 31--48.

\bibitem[B\"ur16]{burg}
P. B\"urgisser, Permanent versus determinant, obstructions, and Kronecker coeffcients, S\'em. Lothar. Combin. {\bf 75} (2016), B75a.

\bibitem[BI17]{bi}
P. B\"urgisser and C. Ikenmeyer, Fundamental invariants of orbit closures, J. Algebra {\bf 477} (2017), 390--434.

\bibitem[BGO+17]{widg}
P. B\"urgisser, A. Garg, R. Oliveira, M. Walter, and A. Wigderson, Alternating minimization, scaling algorithms, and the null-cone problem from invariant theory, arXiv:1711.08039 (2017).

\bibitem[BFG+19]{widg2}
P. B\"urgisser, C. Franks, A. Garg, R. Oliveira, M. Walter, A. Wigderson, 
Towards a theory of non-commutative optimization: geodesic first and second order methods for moment maps and polytopes, 60th IEEE Symposium on Foundations of Computer Science (FOCS) (2019), 845--861, arXiv:1910.12375 (long version).

\bibitem[Cay43]{cay}
A. Cayley, On the theory of determinants, Trans. Cambridge Phil. Soc. VIII (1843), 1--16.

\bibitem[Cay45]{cay2}
A. Cayley, On the theory of linear transformations, Cambridge Math. J. {\bf 4} (1845), 193--209.

\bibitem[CHM09]{chm}
M. Christandl, A. W. Harrow, G. Mitchison, Nonzero Kronecker coefficients and what they tell us about spectra, Comm. Math. Phys. {\bf 270} (2007), 575--585.

\bibitem[Der01]{derk}
H. Derksen, Polynomial bounds for rings of invariants, Proc. Amer. Math. Soc. {\bf 129} (2001), 955--963.

\bibitem[DM20]{dm}
H. Derksen, and V. Makam, An exponential lower bound for the degrees of invariants of cubic forms and tensor actions, Adv. Math. {\bf 368} (2020), Article 107136.

\bibitem[Dri97]{dri1}
A. Drisko, On the number of even and odd Latin squares of order $p + 1$, Adv. Math. {\bf 128} (1997), 20--35.

\bibitem[Dvi93]{dvir93} 
Y. Dvir, On the Kronecker product of $S_{n}$ characters. J. Algebra {\bf 154} (1993), 125--140.

\bibitem[GKZ94]{gkz}
I. M. Gelfand, M. M. Kapranov, A. V. Zelevinsky, Hyperdeterminants. In Discriminants, Resultants, and Multidimensional Determinants, Birkh\"auser, Boston MA, 1994, 444--479.

\bibitem[Gly10]{glynn}
D. Glynn, The conjectures of Alon--Tarsi and Rota in dimension prime minus one, SIAM J. Discrete Math. {\bf 24} (2010), 394--399.

\bibitem[Hil90]{hilb1}
D. Hilbert, \"Uber die Theorie der algebraischen Formen, Math. Ann. {\bf 36} (1890), 473--534.

\bibitem[Hil93]{hilb2}
D. Hilbert, \"Uber die vollen Invariantensysteme, Math. Ann. {\bf 42} (1893), 313--373.

\bibitem[HR93]{hr}
R. Huang and G.-C. Rota, On the relations of various conjectures on Latin squares and straightening coefficients, Discrete Math. {\bf 128} (1994), 225--236.

\bibitem[Ike13]{iken13} 
C. Ikenmeyer, Geometric complexity theory, tensor rank, and Littlewood-Richardson coefficients, PhD diss., Universität Paderborn, 2013.

\bibitem[IMW17]{imw17} 
C. Ikenmeyer, K. D. Mulmuley, M. Walter, On vanishing of Kronecker coefficients, Comput. Complexity {\bf 26} (2017), 949--992. 

\bibitem[Jan95]{jan95} 
J. C. M. Janssen, On even and odd Latin squares, J. Combin. Theory Ser. A {\bf 69} (1995), 173--181. 

\bibitem[Kum15]{kumar}
S. Kumar, A study of the representations supported by the orbit closure of the determinant, Compos. Math. {\bf 151} (2015), 292--312.

\bibitem[LZX21]{LZX21}
X. Li, L. Zhang, H. Xia, Two classes of minimal generic fundamental invariants for tensors, arXiv:2111.07343 (2021).

\bibitem[Luq08]{luq}
J.-G. Luque, Invariants des hypermatrices, Diss. Universit\'e de Marne la Vall\'ee, 2008.	

\bibitem[OEIS]{oeis}
Online Encyclopedia of Integers Sequences, the sequence A176097 at \url{https://oeis.org/A176097}

\bibitem[Val09]{val09}
E. Vallejo, A stability property for coefficients in Kronecker products of complex $S_{n}$ characters. Electron. J. Combin. {\bf 16} (2009), N 22.



\end{thebibliography}
\end{document}